\documentclass[a4paper]{amsart}
\usepackage[utf8]{inputenc}
\usepackage[T1]{fontenc}
\usepackage{lmodern}
\usepackage{amssymb}
\usepackage[all]{xy}
\usepackage{nicefrac,mathtools,enumitem}
\usepackage{microtype}
\usepackage{ifthen}
\usepackage{leftidx}

\usepackage{tikz}
\usetikzlibrary{matrix}
\tikzset{cd/.style=matrix of math nodes,row sep=2em,column sep=2em, text height=1.5ex, text depth=0.5ex}
\tikzset{cdar/.style=->,auto}

\setlist[enumerate,1]{label=\textup{(\arabic*)}}
\setlist[enumerate,2]{label=\textup{(\alph*)}}

\usepackage{xcolor}

\usepackage[pdftitle={Quantum group-twisted tensor products of C*-algebras},
pdfauthor={Ralf Meyer, Sutanu Roy and Stanislaw Lech Woronowicz},
pdfsubject={Mathematics}
]{hyperref}
\usepackage[lite]{amsrefs}
\newcommand*{\MRref}[2]{ \href{http://www.ams.org/mathscinet-getitem?mr=#1}{MR #1}}

\renewcommand{\PrintDOI}[1]{\href{http://dx.doi.org/\detokenize{#1}}{doi: \detokenize{#1}}%
  \IfEmptyBibField{pages}{, (to appear in print)}{}}

\BibSpec{book}{%
    +{}  {\PrintPrimary}                {transition}
    +{,} { \textit}                     {title}
    +{.} { }                            {part}
    +{:} { \textit}                     {subtitle}
    +{,} { \PrintEdition}               {edition}
    +{}  { \PrintEditorsB}              {editor}
    +{,} { \PrintTranslatorsC}          {translator}
    +{,} { \PrintContributions}         {contribution}
    +{,} { }                            {series}
    +{,} { \voltext}                    {volume}
    +{,} { }                            {publisher}
    +{,} { }                            {organization}
    +{,} { }                            {address}
    +{,} { \PrintDateB}                 {date}
    +{,} { }                            {status}
    +{}  { \parenthesize}               {language}
    +{}  { \PrintTranslation}           {translation}
    +{;} { \PrintReprint}               {reprint}
    +{.} { }                            {note}
    +{.} {}                             {transition}
    +{,} { \PrintDOI}                   {doi}
    +{}  {\SentenceSpace \PrintReviews} {review}
}

\numberwithin{equation}{section}

\theoremstyle{plain}
\newtheorem{theorem}[equation]{Theorem}
\newtheorem{lemma}[equation]{Lemma}
\newtheorem{proposition}[equation]{Proposition}

\newtheorem{corollary}[equation]{Corollary}

\theoremstyle{definition}
\newtheorem{definition}[equation]{Definition}

\theoremstyle{remark}
\newtheorem{remark}[equation]{Remark}

\newtheorem{example}[equation]{Example}

\newcommand*{\prto}{\twoheadrightarrow}
\newcommand*{\into}{\rightarrowtail}

\newcommand*{\nb}{\nobreakdash}
\newcommand*{\Star}{\(^*\)\nb-}

\newcommand*{\C}{\mathbb C}
\newcommand*{\Z}{\mathbb Z}
\newcommand*{\R}{\mathbb R}

\newcommand*{\T}{\mathbb T}
\newcommand*{\Mat}{\mathbb M}

\newcommand*{\G}[1][G]{\mathbb #1}
\newcommand*{\DuG}[1][G]{\hat{\mathbb{#1}}}

\newcommand*{\Comult}[1][]{\Delta_{#1}}
\newcommand*{\DuComult}[1][]{\hat{\Delta}_{#1}}
\newcommand*{\Qgrp}[2]{\mathbb #1=(#2,\Comult[#2])}
\newcommand*{\DuQgrp}[2]{\widehat{\mathbb #1}=(\hat{#2},\DuComult[#2])}
\newcommand*{\Coinv}{\textup R}
\newcommand*{\CLS}{\mathrm{CLS}}

\newcommand*{\Bound}{\mathbb B}
\newcommand*{\Comp}{\mathbb K}
\newcommand*{\blank}{\textup{\textvisiblespace}}
\newcommand*{\hboxt}{\mathbin{\widetilde\boxtimes}}
\newcommand*{\hot}{\mathbin{\hat\otimes}}
\newcommand*{\hodot}{\mathbin{\hat\odot}}

\newcommand*{\ima}{\textup i}
\newcommand*{\univ}{\textup u}
\newcommand*{\red}{\textup r}
\newcommand*{\ab}{\textup{ab}}
\newcommand*{\transpose}{\mathsf T}
\newcommand*{\Ad}{\mathrm{Ad}}

\newcommand*{\Contvin}{\textup C_0}
\newcommand*{\Cont}{\textup C}
\newcommand*{\Mor}{\textup{Mor}}
\newcommand*{\Id}{\textup{id}}

\newcommand*{\Multunit}{\mathbb W}
\newcommand*{\multunit}[1][]{\textup W^{#1}}
\newcommand*{\DuMultunit}{\widehat{\mathbb W}}
\newcommand*{\Dumultunit}[1][]{\widehat{\textup W}{}^{#1}}

\newcommand*{\bichar}{\chi}
\newcommand*{\Dubichar}{\hat{\chi}}

\newcommand*{\HeisPair}[1]{(#1,\hat{#1})}

\newcommand*{\dumaxcorep}[1][]{\mathcal W^{#1}}
\newcommand*{\maxcorep}[1][]{\widetilde{\mathcal W}^{#1}}

\newcommand*{\Flip}{\Sigma}
\newcommand*{\flip}{\sigma}
\newcommand*{\Cst}{\mathrm C^*}
\newcommand*{\Cred}{\mathrm C^*_\mathrm r}

\newcommand*{\Cstcat}{\mathfrak{C^*alg}}
\newcommand*{\Forget}{\mathsf{For}}

\newcommand*{\Hils}[1][H]{\mathcal #1}
\newcommand*{\Hilm}[1][E]{\mathcal #1}
\newcommand*{\Mult}{\mathcal M}
\newcommand*{\U}{\mathcal U}

\newcommand*{\defeq}{\mathrel{\vcentcolon=}}

\newcommand*{\norm}[1]{\lVert#1\rVert}
\newcommand*{\conj}[1]{\overline{#1}}

\DeclareMathOperator{\Aut}{Aut}

\begin{document}
\title{Quantum group-twisted tensor products of \(\Cst\)\nb-algebras}

\author{Ralf Meyer}
\email{rameyer@uni-math.gwdg.de}

\author{Sutanu Roy}
\email{sutanu@uni-math.gwdg.de}
\address{Mathematisches Institut\\
  Georg-August Universität Göttingen\\
  Bunsenstraße 3--5\\
  37073 Göttingen\\
  Germany}

\author{Stanisław Lech Woronowicz}
\email{Stanislaw.Woronowicz@fuw.edu.pl}
\address{Instytut Matematyki\\Uniwersytet w Białymstoku, and\\
  Katedra Metod Matematycznych Fizyki\\
  Wydział Fizyki, Uniwersytet Warszawski\\
  Hoża 74\\
  00-682 Warszawa\\
  Poland}

\begin{abstract}
  We put two \(\Cst\)\nb-algebras together in a noncommutative tensor
  product using quantum group coactions on them and a bicharacter
  relating the two quantum groups that act.  We describe this twisted
  tensor product in two equivalent ways, based on certain pairs of
  quantum group representations and based on covariant Hilbert space
  representations, respectively.  We establish basic properties of the
  twisted tensor product and study some examples.
\end{abstract}


\thanks{Supported by the German Research Foundation (Deutsche
  Forschungsgemeinschaft (DFG)) through the Research Training Group
  1493 and the Institutional Strategy of the University of
  G\"ottingen, and by the Alexander von Humboldt-Stiftung
  and the National Science Centre (NCN) grant no.\@ 2011/01/B/ST1/05011.}

\maketitle

\section{Introduction}
\label{sec:introduction}

Several important constructions put together two \(\Cst\)\nb-algebras
in a kind of tensor product where the tensor factors do not commute.
For instance, a noncommutative two-torus is obtained in this way from
two copies of \(\Cont(\T)\).  The reduced crossed product \(A\rtimes_\alpha
G\) for a continuous action \(\alpha\colon G\to\Aut(A)\) of a locally
compact group~\(G\) combines \(A\) and the reduced group \(\Cst\)\nb-algebra
of~\(G\).  Such crossed products also exist for locally compact
quantum groups.  Another example is the skew-commutative tensor
product for \(\Z/2\)-graded \(\Cst\)\nb-algebras, which is defined so
that the odd elements anticommute.

We shall construct twisted tensor products using quantum group
coactions on the tensor factors.  The examples mentioned above are
special cases of our theory.  Our construction is closely related to
one by Vaes~\cite{Vaes:Induction_Imprimitivity}, studied in more
detail by Nest and Voigt~\cite{Nest-Voigt:Poincare}; it is more general
because we allow two different quantum groups to act on the tensor
factors and do not need Haar weights on quantum groups.  Moreover,
we provide two different constructions of the noncommutative tensor
product and use them to prove many formal properties.

Our twisted tensor product uses the following data: two
\(\Cst\)\nb-quantum groups \(\Qgrp{G}{A}\) and \(\Qgrp{H}{B}\) (in the
sense of~\cite{Soltan-Woronowicz:Multiplicative_unitaries}); a
bicharacter \(\bichar\in\U(\hat{A}\otimes\hat{B})\); and two
\(\Cst\)\nb-algebras \(C\) and~\(D\) with continuous coactions
\(\gamma\colon C\to C\otimes A\) and \(\delta\colon D\to D\otimes B\)
of \(\G\) and~\(\G[H]\), respectively.  Then we define a \(\Cst\)\nb-algebra
\[
C\boxtimes_\bichar D
= (C,\gamma) \boxtimes_\bichar (D,\delta)
\]
with nondegenerate \Star{}homomorphisms
\[
C \xrightarrow{\iota_C} \Mult(C\boxtimes_\bichar D)
\xleftarrow{\iota_D} D
\]
such that \(\iota_C(C)\cdot\iota_D(D)\) is linearly dense
in~\(C\boxtimes_\bichar D\).  We briefly call \((C\boxtimes_\bichar
D,\iota_C,\iota_D)\) a crossed product of \(C\) and~\(D\).

We now give several examples.

First the trivial, commutative case.  If \(\bichar=1\) or if
\(\gamma\) or~\(\delta\) is trivial, then \(C\boxtimes_\bichar D\) is
the minimal \(\Cst\)\nb-tensor product with the usual maps \(\iota_C\)
and~\(\iota_D\).

Secondly, let \(A=B=\Cst(\Z/2)\).  Then \(C\) and~\(D\) are
\(\Z/2\)\nb-graded \(\Cst\)\nb-algebras.  Let~\(\bichar\) be the
unique non-trivial bicharacter in
\(\hat{A}\otimes\hat{B}=\Cont(\Z/2\times\Z/2)\), defined by
\(\bichar(a,b)\defeq a\cdot b\) for \(a,b\in\Z/2=\{\pm1\}\).  Then
\(C\boxtimes_\bichar D\) is the (spatial) skew-commutative tensor
product of \(C\) and~\(D\).

Thirdly, let \(\G[H]=\DuG\) be the reduced dual of~\(\G\) and let
\(\bichar=\multunit[A]\in\U(\hat{A}\otimes A)\) be the reduced
bicharacter; here we identify the bidual
of~\(\G\) with~\(\G\).  If~\(\G\) has a Haar weight then our construction 
is equivalent to one by Nest and Voigt~\cite{Nest-Voigt:Poincare}.  
In particular, for~\(D=\hat{A}\) and \(\delta=\DuComult[A]\),
\(C\boxtimes_{\multunit[A]} \hat{A}\) is the reduced crossed product
for the coaction~\(\gamma\).

Finally, let \(A=B=\Cont(\T^n)\), so that coactions of \(\G=\G[H]\)
are actions of the \(n\)\nb-torus group~\(\T^n\), and let
\(C=D=\Cont(\T^n)\) with \(\gamma=\delta=\Comult[A]\), corresponding
to the translation action of~\(\T^n\) on itself.  A bicharacter
\(\bichar\in\U(\hat{A}\otimes\hat{B})\) is equivalent to a map
\(\chi\colon \Z^n\times\Z^n\to\T\) that is multiplicative in both
variables.  Thus \(\chi((a_n),(b_n)) = \prod_{i,j=1}^n
\lambda_{ij}^{a_i\cdot b_j}\) for some \((\lambda_{ij})_{1\le i,j\le
  n}\in\T\).  The resulting tensor product
\(\Cont(\T^n)\boxtimes_\bichar \Cont(\T^n)\) is generated by
\(2n\)~unitaries \(U_1,\dotsc,U_n\) and \(V_1,\dotsc,V_n\), with the
following commutation relations.  First, the \(U_i\) and the~\(V_i\)
commute among themselves, so that they generate two copies of
\(\Cont(\T^n)\).  Secondly, \(V_iU_j = \lambda_{ij} U_j V_i\) for all
\(1\le i,j\le n\).  We get all noncommutative \(2n\)\nb-tori in
this way.

Now we describe two constructions of \(C\boxtimes_\bichar D\).

The first one uses a pair of representations \(\alpha\colon
A\to\Bound(\Hils)\), \(\beta\colon B\to\Bound(\Hils)\) on the same
Hilbert space.  This yields embeddings \(\iota_C\defeq
((\Id_C\otimes\alpha)\circ \gamma)_{13}\colon C\to\Mult(C\otimes
D\otimes\Comp(\Hils))\) and \(\iota_D\defeq ((\Id\otimes\beta)\circ
\delta)_{23}\colon D\to\Mult(C\otimes D\otimes\Comp(\Hils))\).  We let
\(C\boxtimes_\bichar D\) be the closed linear span of
\(\iota_C(C)\cdot \iota_D(D)\) for a suitable choice of \(\alpha\)
and~\(\beta\).  We call suitable pairs \((\alpha,\beta)\)
\emph{\(\bichar\)\nb-Heisenberg pairs}.  The closed linear span of
\(\iota_C(C)\cdot \iota_D(D)\) is a \(\Cst\)\nb-algebra, and different
\(\bichar\)\nb-Heisenberg pairs \((\alpha,\beta)\) yield equivalent
crossed products.

The definition of a \(\bichar\)\nb-Heisenberg pair generalises the
Weyl form of the canonical commutation relations (see
Example~\ref{exa:Heisenberg_pair_name}).  It is also a variant of the
usual pentagon equation for multiplicative unitaries (see
Example~\ref{exa:Heisenberg_from_Multunit}).  In terms of the reduced
bicharacters \(\multunit[A]\in\U(\hat{A}\otimes A)\) and
\(\multunit[B]\in\U(\hat{B}\otimes B)\), the pair
\((\alpha,\beta)\) is a \(\bichar\)\nb-Heisenberg pair if
\(\multunit[A]_{1\alpha}\multunit[B]_{2\beta}
=\multunit[B]_{2\beta}\multunit[A]_{1\alpha} \bichar_{12}\) in
\(\U(\hat{A}\otimes\hat{B}\otimes\Comp(\Hils))\);
here~\(\multunit[A]_{1\alpha}\) means that we apply
\(\Id\otimes\Id\otimes\alpha\) to~\(\multunit[A]_{13}\).

Our second approach uses covariant representations of \((C,\gamma)\)
and~\((D,\delta)\) on Hilbert spaces \(\Hils\) and~\(\Hils[K]\).
These contain corepresentations of \(A\) and~\(B\), which allow us to
turn~\(\bichar\) into a unitary operator~\(Z\) on
\(\Hils\otimes\Hils[K]\).  Assuming that the representations of \(C\)
and~\(D\) are faithful, we show that we get a faithful representation
\(C\boxtimes_\bichar D\to\Bound(\Hils\otimes\Hils[K])\), mapping
\(C\ni c\mapsto c\otimes 1\) and \(D\ni d\mapsto Z(1\otimes d)Z^*\).

We also establish functoriality properties of~\(\boxtimes\).  These
say that a pair of equivariant ``maps'' \(f\colon C\to C'\), \(g\colon
D\to D'\) induces a ``map'' \(C\boxtimes_\bichar D\to
C'\boxtimes_\bichar D'\).  ``Maps'' could mean, among others,
morphisms, \Star{}homomorphisms, completely positive contractions, or
\(\Cst\)\nb-correspondences.  We also examine when
\(f\boxtimes_\bichar g\) is injective or surjective.  Functoriality
for \(\Cst\)\nb-correspondences also shows that \(C\boxtimes_\bichar
D\) is Morita--Rieffel equivalent to~\(C'\boxtimes_\bichar D'\) if
\(C,C'\) and \(D,D'\) are equivariantly Morita--Rieffel equivalent.

Related to Morita--Rieffel equivalence, we show that
\[
(C,\gamma) \boxtimes_\bichar (D,\delta) \cong
(C,\gamma') \boxtimes_\bichar (D,\delta')
\]
if \(\gamma' = \Ad_{u_\gamma} \circ \gamma\) and \(\delta' =
\Ad_{u_\delta} \circ \delta\) for cocycles
\(u_\gamma\in\U(C\otimes A)\) and \(u_\delta\in\U(D\otimes
B)\).  This generalises the well-known isomorphism between the reduced
crossed products for an inner action and the trivial action.

Finally, we consider the examples mentioned above.  If \(A\)
and~\(B\) are Abelian groups, then we identify \(C \boxtimes_\bichar
D\) with a Rieffel deformation of the ordinary tensor product
\(C\otimes D\).  We identify \((C,\gamma)\boxtimes_{\multunit[A]}
(\hat{A},\DuComult[A])\) with the reduced crossed product for the
coaction~\(\gamma\) on~\(C\).

The dual coaction on reduced crossed products is an instance of the
functoriality of~\(\boxtimes\).  The coaction \(\DuComult[A]\colon
\hat{A}\to \hat{A}\otimes\hat{A}\) is a \(\DuG\)\nb-equivariant map
if \(\hat{A}\otimes\hat{A}\) caries the right \(\DuG\)\nb-coaction
\(\Id_{\hat{A}}\otimes \DuComult[A]\).  By functoriality
of~\(\boxtimes\), it induces a morphism
\[
(C,\gamma)\boxtimes_{\multunit[A]} (\hat{A},\DuComult[A])
\to (C,\gamma)\boxtimes_{\multunit[A]}
(\hat{A}\otimes\hat{A},\Id_{\hat{A}}\otimes \DuComult[A])
\cong \hat{A}\otimes \bigl((C,\gamma)\boxtimes_{\multunit[A]}
(\hat{A},\DuComult[A])\bigr).
\]
This is a continuous left \(\hat{A}\)\nb-coaction on
\((C,\gamma)\boxtimes_{\multunit[A]} (\hat{A},\DuComult[A])\).  It
is equivalent to the dual coaction on the reduced crossed product.

In a purely algebraic setting, noncommutative tensor products of two
algebras \(A\) and~\(B\) may be studied using a commutation map
\(R\colon B\otimes A\to A\otimes B\), such that \((a_1\otimes
b_1)\cdot (a_2\otimes b_2)\defeq a_1\cdot R(b_1\otimes a_2)\cdot
b_2\in A\otimes B\); the properties needed for this to be
associative are worked out in~\cite{Daele-Keer:Yang-Baxter}.  Even
more general twisting procedures work quite well algebraically,
see~\cite{Lopez-Panaite-Oystaeyen:Twisting}.  The closest precursor
of our construction~\(\boxtimes\) is
\cite{Majid:Quantum_grp}*{Corollary 9.2.13}, which defines a
noncommutative tensor product for \(H\)\nb-comodule algebras over a
quasitriangular Hopf algebra~\(H\).  In a \(\Cst\)\nb-algebra
context, Ruy Exel~\cite{Exel:Blend_Alloys} treats some examples of
noncommutative tensor products using commutation maps.  Mostly,
however, commutation maps or explicit formulas for the product do
not help to construct \emph{\(\Cst\)\nb-algebras}.  First,
\(R(b\otimes a)=b\cdot a\) need not be a \emph{finite} linear
combination of products \(a_i\cdot b_i\); secondly, we do not yet
know the \(\Cst\)\nb-norm on \(A\otimes B\) in which we could
approximate \(b\cdot a\) by finite sums of products \(a_i\cdot
b_i\); thirdly, the choice of a \(\Cst\)\nb-norm for this
convergence would already impose some subtle constraints on the
possible twisted multiplications (recall, for instance, that the
spectral radius of a self-adjoint element in a \(\Cst\)\nb-algebra
is equal to its norm).  This article addresses the analytic
difficulties in defining noncommutative \(\Cst\)\nb-algebra tensor
products.

\section{Preliminaries}
\label{sec:preliminaries}

For two norm-closed subsets \(X\) and~\(Y\) of a \(\Cst\)\nb-algebra,
let
\[
X\cdot Y\defeq\{xy: x\in X, y\in Y\}^\CLS,
\]
where CLS stands for the \emph{closed linear span}.

For a \(\Cst\)\nb-algebra~\(A\), let \(\Mult(A)\) be its multiplier
algebra and let \(\U(A)\) be the group of unitary multipliers
of~\(A\).  A unitary \(U\in\U(A)\) defines an automorphism
\(\Ad_U\colon A\to A\), \(a\mapsto UaU^*\).

Let~\(\Cstcat\) be the category of \(\Cst\)\nb-algebras with
nondegenerate \Star{}homomorphisms \(\varphi\colon A\to\Mult(B)\) as
morphisms \(A\to B\); let \(\Mor(A,B)\) denote this set of morphisms.

Let~\(\Hils\) be a Hilbert space.  A \emph{representation} of a
\(\Cst\)\nb-algebra~\(A\) is a nondegenerate
\Star{}homomorphism \(A\to\Bound(\Hils)\).  Since
\(\Bound(\Hils)=\Mult(\Comp(\Hils))\) and the nondegeneracy
conditions \(A\cdot\Comp(\Hils)=\Comp(\Hils)\) and
\(A\cdot\Hils=\Hils\) are equivalent, this is the same as a
morphism from~\(A\) to~\(\Comp(\Hils)\).

We write~\(\Flip\) for the tensor flip \(\Hils\otimes\Hils[K]\to
\Hils[K]\otimes\Hils\), \(x\otimes y\mapsto y\otimes x\), for two
Hilbert spaces \(\Hils\) and~\(\Hils[K]\).  We write~\(\flip\) for the
tensor flip isomorphism \(A\otimes B\to B\otimes A\) for two
\(\Cst\)\nb-algebras \(A\) and~\(B\).

\subsection{Crossed tensor products}
\label{sec:crossed_tensor}

\begin{definition}[compare~\cite{Woronowicz:Braided_Qnt_Grp}]
  \label{def:crossed_product}
  Let \(A\), \(B\), \(C\) be \(\Cst\)\nb-algebras,
  \(\alpha\in\Mor(A,C)\) and \(\beta\in\Mor(B,C)\).  If
  \(\alpha(A)\cdot\beta(B)=C\), then we call \((C,\alpha,\beta)\) a
  \emph{crossed product} or \emph{crossed tensor product} of \(A\)
  and~\(B\).
\end{definition}

\begin{example}
  \label{ex:tensor_product_as_crossed_prod}
  The spatial tensor product \(C=A\otimes B\) of two
  \(\Cst\)\nb-algebras with \(\alpha(a)=a\otimes 1_B\) and \(\beta(b)
  = 1_A\otimes b\) is the simplest example of a crossed product.
  Here \(1_A\in\Mult(A)\) and \(1_B\in\Mult(B)\).
\end{example}

Let \(\alpha\) and~\(\beta\) be (nondegenerate) representations of
\(A\) and~\(B\) on the same Hilbert space~\(\Hils\) such that
\(\alpha(A)\cdot\beta(B)\) and \(\beta(B)\cdot\alpha(A)\) are the
same subspace of~\(\Bound(\Hils)\).  Then \(C\defeq
\alpha(A)\cdot\beta(B)\) is a \(\Cst\)\nb-algebra,
\(\alpha\in\Mor(A, C)\) and \(\beta\in\Mor(B,C)\).  Thus~\(C\) is a
crossed product of \(A\) and~\(B\).  This suggests that crossed
products are defined by some commutation relations between
\(\alpha\) and~\(\beta\).  Analytic difficulties may, however,
prevent us from writing down such commutation relations explicitly
(see our discussion at the end of the introduction).

\begin{definition}
  \label{def:equiv_of_crossed_prod}
  An \emph{equivalence} between two crossed products
  \(C_1=\alpha_1(A)\cdot\beta_1(B)\) and
  \(C_2=\alpha_2(A)\cdot\beta_2(B)\) of \(A\) and~\(B\) is an
  isomorphism \(\varphi\colon C_1\to C_2\) with
  \(\varphi\circ\alpha_1=\alpha_2\) and
  \(\varphi\circ\beta_1=\beta_2\).
\end{definition}

Any faithful morphism \(\varphi\in\Mor(C_1,C_2)\) with
\(\varphi\circ\alpha_1=\alpha_2\) and \(\varphi\circ\beta_1=\beta_2\)
satisfies \(\varphi(C_1)=C_2\) and hence is an equivalence of crossed
products.

\begin{example}
  \label{ex:equiv_via_unitary}
  Let \(C=\alpha(A)\cdot\beta(B)\) be a crossed product and
  \(U\in\U(C)\).  Then
  \[
  (C,\alpha, \beta) \simeq (C,\Ad_{U}\circ\alpha,\Ad_{U}\circ\beta).
  \]
\end{example}

\subsection{Multiplicative unitaries and quantum groups}
\label{sec:multunit_quantum_groups}

\begin{definition}[\cite{Baaj-Skandalis:Unitaires}]
  \label{def:multunit}
  Let~\(\Hils\) be a Hilbert space.  A unitary
  \(\Multunit\in\U(\Hils\otimes\Hils)\) is \emph{multiplicative} if it
  satisfies the \emph{pentagon equation}
  \begin{equation}
    \label{eq:pentagon}
    \Multunit_{23}\Multunit_{12}
    = \Multunit_{12}\Multunit_{13}\Multunit_{23}
    \qquad
    \text{in \(\U(\Hils\otimes\Hils\otimes\Hils).\)}
  \end{equation}
\end{definition}

Technical assumptions such as manageability
(\cite{Woronowicz:Multiplicative_Unitaries_to_Quantum_grp}) or, more
generally, modularity (\cite{Soltan-Woronowicz:Remark_manageable}) are
needed in order to construct \(\Cst\)\nb-algebras out of a
multiplicative unitary.

\begin{theorem}[\cites{Soltan-Woronowicz:Remark_manageable,
    Soltan-Woronowicz:Multiplicative_unitaries,
    Woronowicz:Multiplicative_Unitaries_to_Quantum_grp}]
  \label{the:Cst_quantum_grp_and_mult_unit}
  Let~\(\Hils\) be a separable Hilbert space and
  \(\Multunit\in\U(\Hils\otimes\Hils)\) a modular
  multiplicative unitary.  Let
  \begin{alignat}{2}
    \label{eq:first_leg_slice}
    A &\defeq \{(\omega\otimes\Id_{\Hils})\Multunit :
    \omega\in\Bound(\Hils)_*\}^\CLS,\\
    \label{eq:second_leg_slice}
    \hat{A} &\defeq \{(\Id_{\Hils}\otimes\omega)\Multunit :
    \omega\in\Bound(\Hils)_*\}^\CLS.
  \end{alignat}
  \begin{enumerate}
  \item \(A\) and \(\hat{A}\) are separable, nondegenerate
    \(\Cst\)\nb-subalgebras of~\(\Bound(\Hils)\).
  \item \(\Multunit\in\U(\hat{A}\otimes
    A)\subseteq\U(\Hils\otimes\Hils)\).  We write~\(\multunit[A]\)
    for~\(\Multunit\) viewed as a unitary multiplier of
    \(\hat{A}\otimes A\) and call it \emph{reduced bicharacter}.
  \item There is a unique \(\Comult[A]\in\Mor(A,A\otimes A)\) such
    that
    \begin{equation}
      \label{eq:Comult_W}
      (\Id_{\hat{A}}\otimes \Comult[A])\multunit[A]
      = \multunit[A]_{12}\multunit[A]_{13}
      \qquad \text{in \(\U(\hat{A}\otimes A\otimes A)\);}
    \end{equation}
    it is \emph{coassociative}:
    \begin{equation}
      \label{eq:coassociative}
      (\Comult[A]\otimes\Id_A)\circ\Comult[A]
      = (\Id_A\otimes\Comult[A])\circ\Comult[A],
    \end{equation}
    and satisfies the \emph{Podle\'s condition}
    \begin{equation}
      \label{eq:Podles}
      \Comult[A](A)\cdot(1_A\otimes A)
      = A\otimes A
      = (A\otimes 1_A)\cdot\Comult[A](A).
    \end{equation}
  \item There is a unique ultraweakly continuous, linear
    anti-automorphism~\(\Coinv_A\) of~\(A\) with
    \begin{align}
      \label{eq:opp_comult_via_antipode}
      \Comult[A]\circ \Coinv_A &=
      \flip\circ(\Coinv_A\otimes \Coinv_A)\circ\Comult[A],
    \end{align}
    where \(\flip(x\otimes y)=y\otimes x\).  It satisfies
    \(\Coinv_A^2=\Id_A\).
  \end{enumerate}
\end{theorem}

A \emph{\(\Cst\)\nb-quantum group} is a \(\Cst\)\nb-bialgebra
\(\Qgrp{G}{A}\) constructed from a modular
multiplicative unitary.  We do not need Haar weights.

The \emph{dual} multiplicative unitary is \(\DuMultunit\defeq
\Flip\Multunit^*\Flip\in\U(\Hils\otimes\Hils)\), where
\(\Flip(x\otimes y)=y\otimes x\).  It is modular or manageable
if~\(\Multunit\) is.  The \(\Cst\)\nb-quantum group
\(\DuQgrp{G}{A}\) generated by~\(\DuMultunit\) is the \emph{dual}
of~\(\G\).  Its comultiplication is characterised by
\begin{equation}
  \label{eq:dual_Comult_W}
  (\DuComult[A]\otimes \Id_A)\multunit[A]
  = \multunit[A]_{23}\multunit[A]_{13}
  \qquad\text{in \(\U(\hat{A}\otimes \hat{A}\otimes A)\).}
\end{equation}

Let \(\Qgrp{G}{A}\) be a \(\Cst\)\nb-quantum group.

\begin{definition}
  \label{def:cont_coaction}
  A \emph{continuous \textup(right\textup) coaction} of~\(\G\) on a
  \(\Cst\)\nb-algebra~\(C\) is a morphism \(\gamma\colon C\to C\otimes
  A\) with the following properties:
  \begin{enumerate}
  \item \(\gamma\) is injective;
  \item \(\gamma\) is a comodule structure, that is,
    \((\Id_C\otimes\Comult[A])\gamma = (\gamma\otimes\Id_A)\gamma\):
    \begin{equation}
      \label{eq:right_coaction}
      \begin{tikzpicture}[baseline=(current bounding box.west)]
        \matrix(m)[cd,column sep=6em]{
          C&C\otimes A\\
          C\otimes A& C\otimes A\otimes A\\
        };
        \draw[cdar] (m-1-1) -- node {\(\gamma\)} (m-1-2);
        \draw[cdar] (m-1-1) -- node[swap] {\(\gamma\)} (m-2-1);
        \draw[cdar] (m-2-1) -- node {\(\gamma\otimes\Id_A\)} (m-2-2);
        \draw[cdar] (m-1-2) -- node {\(\Id_C\otimes\Comult[A]\)} (m-2-2);
      \end{tikzpicture}
    \end{equation}
  \item \(\gamma\) satisfies the \emph{Podleś condition}
    \(\gamma(C)\cdot(1_C\otimes A)=C\otimes A\).
  \end{enumerate}
  We call \((C,\gamma)\) a \emph{\(\G\)\nb-\(\Cst\)\nb-algebra}.  We
  often drop~\(\gamma\) from our notation.

  A morphism \(f\colon C\to D\) between two
  \(\G\)\nb-\(\Cst\)\nb-algebras \((C,\gamma)\) and \((D,\delta)\) is
  \emph{\(\G\)\nb-\hspace{0pt}equivariant} if \(\delta\circ f =
  (f\otimes\Id_A)\circ\gamma\).  Let \(\Mor^{\G}(C,D)\) be the set of
  \(\G\)\nb-equivariant morphisms from~\(C\) to~\(D\).  Let
  \(\Cstcat(\G)\) be the category with \(\G\)\nb-\(\Cst\)-algebras as
  objects and \(\G\)\nb-equivariant morphisms as arrows.
\end{definition}

\begin{example}
  \label{ex:continuous_coactions}
  The \emph{trivial} \(\G\)\nb-coaction on a \(\Cst\)\nb-algebra~\(C\)
  is \(\tau\colon C\to C\otimes A\), \(c\mapsto c\otimes 1_A\).  It is
  always continuous.
  Theorem~\ref{the:Cst_quantum_grp_and_mult_unit}.3 implies that
  \(\Comult[A]\colon A\to A\otimes A\) is a continuous \(\G\)\nb-coaction 
  on~\(A\) for any \(\Cst\)\nb-quantum group~\(\Qgrp{G}{A}\).  More generally,
  \(\Id_C\otimes\Comult[A]\colon C\otimes A\to C\otimes A\otimes A\)
  is a continuous \(\G\)\nb-coaction on \(C\otimes A\) for any
  \(\Cst\)\nb-algebra~\(C\).  The following lemma says that any
  continuous coaction may be embedded into one of this form.
\end{example}

\begin{lemma}
  \label{lemm:G_Cst_alg_indentific}
  Let~\(C\) be a \(\Cst\)\nb-algebra and~\(D\) a
  \(\Cst\)\nb-subalgebra of~\(\Mult(C\otimes A)\) with
  \begin{equation}
    \label{eq:G_cst_sub_alg_cond}
    (\Id_C\otimes\Comult[A])(D)\cdot (1_{C\otimes A}\otimes A)
    = D\otimes A.
  \end{equation}
  Then~\(D\) with the coaction \(\delta\defeq
  (\Id_C\otimes\Comult[A])|_D\colon D\to D\otimes A\) is a
  \(\G\)\nb-\(\Cst\)-algebra, and the embedding \(D\to\Mult(C\otimes
  A)\) is a \(\G\)\nb-equivariant morphism.

  Every \(\G\)\nb-\(\Cst\)-algebra is isomorphic to one of this form.
\end{lemma}

\begin{proof}
  Equation~\eqref{eq:G_cst_sub_alg_cond} implies that
  \(\Id_C\otimes\Comult[A]\) maps~\(D\) into \(\Mult(D\otimes A)\) as
  claimed.  The coaction~\(\delta\) is injective and coassociative
  because \(\Id_C\otimes\Comult[A]\) is, and
  \eqref{eq:G_cst_sub_alg_cond} is the Podle\'s condition
  for~\(\delta\).  Thus~\(\delta\) is a continuous \(\G\)\nb-coaction.
  The embedding is equivariant by construction.

  Now let \((C,\gamma)\) be a \(\G\)\nb-\(\Cst\)-algebra.  Let
  \(D\defeq\gamma(C)\subseteq\Mult(C\otimes A)\).  The comodule
  property~\eqref{eq:right_coaction} and the Podle\'s condition
  for~\(\gamma\) imply that~\(D\)
  satisfies~\eqref{eq:G_cst_sub_alg_cond}:
  \begin{multline*}
    (\Id_C\otimes\Comult[A])\gamma(C)\cdot(1_{C\otimes A}\otimes A)
    = (\gamma\otimes\Id_A)\bigl(\gamma(C)\cdot(1_C\otimes A)\bigr)
    \\ = (\gamma\otimes\Id_A)(C\otimes A)
    = \gamma(C)\otimes A.
  \end{multline*}
  The isomorphism \(\gamma\colon C\to D\) is \(\G\)\nb-equivariant by
  the comodule property~\eqref{eq:right_coaction} of~\(\gamma\).
\end{proof}

\begin{definition}
  \label{def:corepresentation}
  A (right) \emph{corepresentation} of~\(\G\) on a Hilbert
  space~\(\Hils\) is a unitary \(U\in\U(\Comp(\Hils)\otimes A)\)
  with
  \begin{equation}
    \label{eq:corep_cond}
    (\Id_{\Comp(\Hils)}\otimes\Comult[A]) U = U_{12}U_{13}
    \qquad\text{in }\U(\Comp(\Hils)\otimes A\otimes A).
  \end{equation}
\end{definition}

\begin{definition}
  \label{def:covariant_corep}
  A \emph{covariant representation} of \((C,\gamma,A)\) on a Hilbert
  space~\(\Hils\) is a pair consisting of a corepresentation
  \(U\in\U(\Comp(\Hils)\otimes A)\) and a representation
  \(\varphi\colon C\to\Bound(\Hils)\) that satisfy the covariance
  condition
  \begin{equation}
    \label{eq:covariant_corep}
    (\varphi\otimes\Id_A)\circ \gamma(c) =
    U(\varphi(c)\otimes 1_A)U^*
    \qquad\text{in }\U(\Comp(\Hils)\otimes A)
  \end{equation}
  for all \(c\in C\).  A covariant representation is called
  \emph{faithful} if~\(\varphi\) is faithful.
\end{definition}

\subsection{Universal quantum groups}
\label{sec:univ_qgr}

The universal quantum group \(\G^\univ\defeq
(A^\univ,\Comult[A^\univ])\) associated to \(\Qgrp{G}{A}\) is
introduced in~\cite{Soltan-Woronowicz:Multiplicative_unitaries}.  By
construction, it comes with a \emph{reducing map} \(\Lambda\colon
A^\univ\to A\) and a universal bicharacter \(\dumaxcorep[A]\in
\U(\hat{A}\otimes A^\univ)\).  This may also be characterised as the
unique bicharacter in~\(\U(\hat{A}\otimes A^\univ)\) that lifts
\(\multunit[A]\in \U(\hat{A}\otimes A)\) in the sense that
\(\Id_{\hat{A}}\otimes\Lambda(\dumaxcorep[A])=\multunit[A]\).

Similarly, there are unique bicharacters in
\(\U(\hat{A}^\univ\otimes A)\) and \(\U(\hat{A}^\univ\otimes
A^\univ)\) that lift \(\multunit[A]\in \U(\hat{A}\otimes A)\); the
latter is constructed in~\cite{Meyer-Roy-Woronowicz:Homomorphisms}.

The universality of \(\maxcorep[A]\in\U(\hat{A}^\univ\otimes A)\)
implies that for any corepresentation~\(U^{\Hils}\) of~\(\G\) on a
Hilbert space (or Hilbert module)~\(\Hils\), there is a unique
representation \(\rho\colon\hat{A}^\univ\to\Bound(\Hils)\) with
\((\rho\otimes\Id_A)\maxcorep[A]=U^{\Hils}\).

\subsection{Bicharacters as quantum group morphisms}
\label{sec:bicharacters_morphisms}

Let \(\Qgrp{G}{A}\) and \(\Qgrp{H}{B}\) be \(\Cst\)\nb-quantum groups.
Let \(\DuQgrp{G}{A}\) and \(\DuQgrp{H}{B}\) be their duals.

\begin{definition}[\cite{Meyer-Roy-Woronowicz:Homomorphisms}*{Definition
    16}]
  \label{def:bicharacter}
  A \emph{bicharacter from \(\G\) to~\(\DuG[H]\)} is a unitary
  \(\bichar\in\U(\hat{A}\otimes \hat{B})\) with
  \begin{alignat}{2}
    \label{eq:bichar_char_in_first_leg}
    (\DuComult[A]\otimes\Id_{\hat{B}})\bichar
    &=\bichar_{23}\bichar_{13}
    &\qquad &\text{in }
    \U(\hat{A}\otimes\hat{A}\otimes \hat{B}),\\
    \label{eq:bichar_char_in_second_leg}
    (\Id_{\hat{A}}\otimes\DuComult[B])\bichar
    &=\bichar_{12}\bichar_{13}
    &\qquad &\text{in }
    \U(\hat{A}\otimes \hat{B}\otimes \hat{B}).
  \end{alignat}
\end{definition}

Bicharacters in \(\U(\hat{A}\otimes B)\) are interpreted as quantum
group morphisms from~\(\G\) to~\(\G[H]\)
in~\cite{Meyer-Roy-Woronowicz:Homomorphisms}.  We shall use
bicharacters in \(\U(\hat{A}\otimes \hat{B})\) throughout and rewrite
some definitions in~\cite{Meyer-Roy-Woronowicz:Homomorphisms} in this
setting.

\begin{definition}
  \label{def:right_quantum_morphism}
  A \emph{right quantum group morphism} from~\(\G\) to~\(\DuG[H]\) is
  a morphism \(\Delta_R\colon A\to A\otimes\hat{B}\) such that the
  following diagrams commute:
  \begin{equation}
    \label{eq:right_homomorphism}
    \begin{tikzpicture}[baseline=(current bounding box.west)]
      \matrix(m)[cd,column sep=4em]{
        A&A\otimes \hat{B}\\
        A\otimes A& A\otimes A\otimes \hat{B}\\
      };
      \draw[cdar] (m-1-1) -- node {\(\Delta_R\)} (m-1-2);
      \draw[cdar] (m-1-1) -- node[swap] {\(\Comult[A]\)} (m-2-1);
      \draw[cdar] (m-1-2) -- node {\(\Comult[A]\otimes\Id_{\hat{B}}\)} (m-2-2);
      \draw[cdar] (m-2-1) -- node[swap] {\(\Id_A\otimes\Delta_R\)} (m-2-2);
    \end{tikzpicture}
    \quad
    \begin{tikzpicture}[baseline=(current bounding box.west)]
      \matrix(m)[cd,column sep=4em]{
        A&A\otimes \hat{B}\\
        A\otimes \hat{B}& A\otimes \hat{B}\otimes \hat{B}\\
      };
      \draw[cdar] (m-1-1) -- node {\(\Delta_R\)} (m-1-2);
      \draw[cdar] (m-1-1) -- node[swap] {\(\Delta_R\)} (m-2-1);
      \draw[cdar] (m-2-1) -- node[swap] {\(\Delta_R\otimes\Id_{\hat{B}}\)} (m-2-2);
      \draw[cdar] (m-1-2) -- node {\(\Id_A\otimes\DuComult[B]\)} (m-2-2);
    \end{tikzpicture}
  \end{equation}
\end{definition}

The following theorem summarises some of the main results
of~\cite{Meyer-Roy-Woronowicz:Homomorphisms}.

\begin{theorem}
  \label{the:equivalent_notion_of_homomorphisms}
  There are natural bijections between the following sets:
  \begin{enumerate}
  \item bicharacters \(\bichar\in\U(\hat{A}\otimes\hat{B})\)
    from~\(\G\) to~\(\DuG[H]\);
  \item bicharacters \(\Dubichar\in\U(\hat{B}\otimes\hat{A})\)
    from~\(\G[H]\) to~\(\DuG\);
  \item right quantum group homomorphisms \(\Delta_R\colon A\to
    A\otimes \hat{B}\);
  \item functors \(F\colon\Cstcat(\G)\to\Cstcat(\DuG[H])\) with
    \(\Forget_{\DuG[H]}\circ F=\Forget_{\G}\) for the forgetful
    functor \(\Forget_{\G}\colon\Cstcat(\G)\to\Cstcat\);
  \item Hopf \Star{}homomorphisms \(f\colon A^\univ\to\hat{B}^\univ\)
    between universal quantum groups;
  \item bicharacters
    \(\bichar^\univ\in\U(\hat{A}^\univ\otimes\hat{B}^\univ)\).
  \end{enumerate}
  The first bijection maps a bicharacter~\(\bichar\) to
  \begin{equation}
    \label{eq:dual_bicharacter}
    \Dubichar\defeq\flip(\bichar^*).
  \end{equation}
  A bicharacter~\(\bichar\) and a right quantum group
  homomorphism~\(\Delta_R\) determine each other uniquely via
  \begin{equation}
    \label{eq:def_V_via_right_homomorphism}
    (\Id_{\hat{A}} \otimes \Delta_R)(\multunit[A])
    = \multunit[A]_{12}\bichar_{13}.
  \end{equation}
  The functor~\(F\) associated to~\(\Delta_R\) is the unique one that
  maps \((A,\Comult[A])\) to \((A,\Delta_R)\).  In general, \(F\) maps
  a continuous \(\G\)\nb-coaction \(\gamma\colon C\to C\otimes A\) to
  the unique \(\DuG[H]\)\nb-coaction \(\delta\colon C\to C\otimes
  \hat{B}\) for which the following diagram commutes:
  \begin{equation}
    \label{eq:right_quantum_group_homomorphism_as_fucntor}
    \begin{tikzpicture}[baseline=(current bounding box.west)]
      \matrix(m)[cd,column sep=4.5em]{
        C&C\otimes A\\
        C\otimes\hat{B}& C\otimes A\otimes \hat{B}\\
      };
      \draw[cdar] (m-1-1) -- node {\(\gamma\)} (m-1-2);
      \draw[cdar] (m-1-1) -- node[swap] {\(\delta\)} (m-2-1);
      \draw[cdar] (m-1-2) -- node {\(\Id_C\otimes\Delta_R\)} (m-2-2);
      \draw[cdar] (m-2-1) -- node[swap] {\(\gamma\otimes\Id_{\hat{B}}\)} (m-2-2);
    \end{tikzpicture}
  \end{equation}
  The bicharacter in \(\U(\hat{A}\otimes\hat{B})\)
  associated to a Hopf \Star{}homomorphism \(f\colon A^\univ\to
  \hat{B}^\univ\) is \(\bichar \defeq (\Id_{\hat{A}}\otimes
  \Lambda_{\hat{B}}f)(\dumaxcorep[A])\), where
  \(\dumaxcorep[A]\in\U(\hat{A}\otimes A^\univ)\) is the
  unique bicharacter lifting \(\multunit[A]\in
  \U(\hat{A}\otimes A)\) and \(\Lambda_{\hat{B}}\colon
  \hat{B}^\univ\to\hat{B}\) is the reducing map.
\end{theorem}

\section{Heisenberg pairs and twisted tensor products}
\label{sec:Heisenberg_pair}

This section introduces Heisenberg and anti-Heisenberg pairs
and uses them to construct our noncommutative tensor product,
after establishing properties of Heisenberg pairs necessary for
that purpose.

Let \(\Qgrp{G}{A}\) and~\(\Qgrp{H}{B}\) be \(\Cst\)\nb-quantum
groups.  Let \(\multunit[A]\in\U(\hat{A}\otimes A)\) and
\(\multunit[B]\in\U(\hat{B}\otimes B)\) be their reduced
bicharacters.  Let \(\bichar\in\U(\hat{A}\otimes\hat{B})\)
be a bicharacter from \(A\) to~\(\hat{B}\).  Heisenberg pairs
and anti-Heisenberg pairs are pairs of
representations~\((\alpha,\beta)\) of \(\hat{A}\) and~\(\hat{B}\)
on the same Hilbert space~\(\Hils\) that satisfy suitable
commutation relation.

We use these pairs to define twisted tensor products
\(C\boxtimes_\bichar D\) in Section~\ref{sec:boxtimes_via_Heisenberg}.
A crucial technical point is to show that a pair of representations of
\(C\) and~\(D\) generates a crossed product \(\Cst\)\nb-algebra.  Here
the commutativity result in Section~\ref{sec:commutativity} is
crucial.  In addition, we construct examples of
\(\bichar\)\nb-Heisenberg and \(\bichar\)\nb-anti-Heisenberg pairs,
thus proving their existence, and characterise them in equivalent
ways.

\begin{definition}
  \label{def:V-Heisenberg_pair}
  A pair of representations \(\alpha\colon A\to\Bound(\Hils)\),
  \(\beta\colon B\to\Bound(\Hils)\) is called a
  \emph{\(\bichar\)\nb-Heisenberg pair} or briefly
  \emph{Heisenberg pair} if
  \begin{equation}
    \label{eq:V-Heisenberg_pair}
    \multunit[A]_{1\alpha}\multunit[B]_{2\beta}
    =\multunit[B]_{2\beta}\multunit[A]_{1\alpha} \bichar_{12}
    \qquad\text{in }\U(\hat{A}\otimes\hat{B}\otimes\Comp(\Hils));
  \end{equation}
  here \(\multunit[A]_{1\alpha} \defeq
  ((\Id_{\hat{A}}\otimes\alpha)\multunit[A])_{13}\) and
  \(\multunit[B]_{2\beta} \defeq
  ((\Id_{\hat{B}}\otimes\beta)\multunit[B])_{23}\).  It is
  called a \emph{\(\bichar\)\nb-anti-Heisenberg pair} or
  briefly \emph{anti-Heisenberg pair} if
  \begin{equation}
    \label{eq:V-anti-Heisenberg_pair}
    \multunit[B]_{2\beta}\multunit[A]_{1\alpha}
    =\bichar_{12}\multunit[A]_{1\alpha}\multunit[B]_{2\beta}
    \qquad\text{in \(\U(\hat{A}\otimes\hat{B}\otimes\Comp(\Hils))\),}
  \end{equation}
  with similar conventions as above.
\end{definition}

We name these pairs after Heisenberg because of the following example:

\begin{example}
  \label{exa:Heisenberg_pair_name}
  Let \(A=B=\Cst(\R)\) be the group~\(\R\) viewed as a quantum group,
  and let \(\bichar\in\U(\hat{A}\otimes\hat{B})\cong
  \Cont(\R\times\R,\T)\) be the standard bicharacter \((s,t)\mapsto
  \exp(\ima st)\).  A pair of representations of~\(A\) is a pair of
  unitary one-parameter groups \((U_1(s),U_2(t))_{s,t\in\R}\).
  Equation~\eqref{eq:V-Heisenberg_pair} is equivalent to the canonical
  commutation relation in the \emph{Weyl form}:
  \[
  U_2(t)U_1(s) =\exp(-\ima st) U_1(s)U_2(t)\qquad
  \text{for all } s,t\in\R.
  \]
\end{example}

The case where \(\G[H]=\DuG\) and
\(\bichar=\multunit[A]\in\U(\hat{A}\otimes A)\) is the reduced
bicharacter of~\(\G\) is particularly interesting:

\begin{definition}
  \label{def:G_Heisenberg_pair}
  A \(\multunit[A]\)\nb-Heisenberg or
  \(\multunit[A]\)\nb-anti-Heisenberg pair is also called a
  \emph{\(\G\)\nb-Heisenberg pair} or
  \emph{\(\G\)\nb-anti-Heisenberg pair}, respectively.
\end{definition}

\begin{lemma}
  \label{lem:Heisenberg_pair_pentagon}
  A pair of representations~\((\pi,\hat{\pi})\) of \(A\)
  and~\(\hat{A}\) on~\(\Hils\) is a \(\G\)\nb-Heisenberg pair
  if and only if
  \begin{equation}
    \label{eq:Heisenberg_pair}
    \multunit[A]_{\hat{\pi}3}\multunit[A]_{1\pi}
    = \multunit[A]_{1\pi}\multunit[A]_{13}\multunit[A]_{\hat{\pi}3}
    \qquad\text{in }\U(\hat{A}\otimes\Comp(\Hils)\otimes A).
  \end{equation}
  It is a \(\G\)\nb-anti-Heisenberg pair if and only if
  \begin{equation}
    \label{eq:anti-Heisenberg_pair}
    \multunit[A]_{1\pi}\multunit[A]_{\hat{\pi}3}
    = \multunit[A]_{\hat{\pi}3}\multunit[A]_{13}\multunit[A]_{1\pi}
    \qquad\text{in }\U(\hat{A}\otimes\Comp(\Hils)\otimes A).
  \end{equation}
\end{lemma}

\begin{proof}
  Let \(\pi\) and~\(\hat{\pi}\) be representations of \(A\)
  and~\(\hat{A}\) on~\(\Hils\)
  satisfying~\eqref{eq:Heisenberg_pair}.  When we
  apply~\(\flip_{23}\) to both sides
  of~\eqref{eq:Heisenberg_pair} we get
  \[
  (\Dumultunit[A]_{2\hat{\pi}})^*\multunit[A]_{1\pi}
  = \multunit[A]_{1\pi}\multunit[A]_{12}(\Dumultunit[A]_{2\hat{\pi}})^*
  \qquad\text{in }\U(\hat{A}\otimes A\otimes\Comp(\Hils)).
  \]
  This is equivalent to
  \(\multunit[A]_{1\pi}\Dumultunit[A]_{2\hat{\pi}} =
  \Dumultunit[A]_{2\hat{\pi}}\multunit[A]_{1\pi}\multunit[A]_{12}\),
  which is~\eqref{eq:V-Heisenberg_pair} for \(\hat{B}=A\),
  \(\bichar=\multunit[A]\), \(\alpha=\pi\) and
  \(\beta=\hat{\pi}\).  This computation may be reversed as
  well.

  The computation for anti-Heisenberg pairs is similar.
\end{proof}

\begin{example}
  \label{exa:Heisenberg_from_Multunit}
  Let \(\Multunit\in\U(\Hils\otimes\Hils)\) be a modular
  multiplicative unitary generating~\(\G\).  Thus there are faithful
  representations \(\pi\colon A\to \Bound(\Hils)\) and \(\hat\pi\colon
  \hat{A}\to \Bound(\Hils)\) with \(\Multunit =
  (\hat\pi\otimes\pi)(\multunit[A])\).  They form a
  \(\G\)\nb-Heisenberg pair: the condition~\eqref{eq:Heisenberg_pair}
  is equivalent to the pentagon equation~\eqref{eq:pentagon}
  for~\(\Multunit\).  Conversely, a pair~\((\pi,\hat\pi)\) of faithful
  representations is a \(\G\)\nb-Heisenberg pair if and only if
  \((\hat\pi\otimes\pi)(\multunit[A])\) is a multiplicative unitary.
\end{example}

Let~\(\conj{\Hils}\) be the conjugate Hilbert space to the Hilbert
space~\(\Hils\).  The \emph{transpose} of an operator
\(x\in\Bound(\Hils)\) is the operator
\(x^\transpose\in\Bound(\conj{\Hils})\) defined by
\(x^\transpose(\conj{\xi}) \defeq \conj{x^*\xi}\) for all
\(\xi\in\Hils\).  The transposition is a linear, involutive
anti-automorphism \(\Bound(\Hils)\to\Bound(\conj{\Hils})\).  The
unitary antipode \(\Coinv_A\colon A\to A\) is also a linear,
involutive anti-automorphism (see
Theorem~\ref{the:Cst_quantum_grp_and_mult_unit}).  Therefore, if
\(\alpha\colon A\to\Bound(\Hils)\) and \(\beta\colon
B\to\Bound(\Hils)\) are representations, then so are
\begin{alignat*}{2}
  \tilde\alpha&\colon A\to\Bound(\conj{\Hils}),&\qquad
  a&\mapsto (\Coinv_A(a))^\transpose,\\
  \tilde\beta&\colon B\to\Bound(\conj{\Hils}),&\qquad
  b&\mapsto (\Coinv_B(b))^\transpose.
\end{alignat*}

\begin{lemma}
  \label{lem:Heisenberg_vs_anti-Heisenberg}
  The pair \((\alpha,\beta)\) is Heisenberg if and only
  if \((\tilde\alpha,\tilde\beta)\) is anti-Heisenberg.
\end{lemma}

\begin{proof}
  Let \((\alpha,\beta)\) be a pair of representations.  We have
  \((\Coinv_{\hat{A}}\otimes \Coinv_{\hat{B}})\bichar =\bichar\) for
  any bicharacter in \(\U(\hat{A}\otimes\hat{B})\) by
  \cite{Meyer-Roy-Woronowicz:Homomorphisms}*{Proposition 3.10}.  In
  particular, this applies to \(\multunit[A]\) and~\(\multunit[B]\).
  Since \(\Coinv_{\hat{A}}\otimes\Coinv_{\hat{B}}\otimes\transpose\)
  is antimultiplicative, we get
  \begin{align*}
    \multunit[B]_{2\tilde\beta}\multunit[A]_{1\tilde\alpha}
    &= (\Coinv_{\hat{A}}\otimes\Coinv_{\hat{B}}\otimes\transpose)
    (\multunit[A]_{1\alpha}\multunit[B]_{2\beta}),\\
    \bichar_{12}\multunit[A]_{1\tilde\alpha}\multunit[B]_{2\tilde\beta}
    &= (\Coinv_{\hat{A}}\otimes\Coinv_{\hat{B}}\otimes\transpose)
    (\multunit[B]_{2\beta}\multunit[A]_{1\alpha}\bichar_{12}).
  \end{align*}
  Since \(\Coinv_{\hat{A}}\otimes\Coinv_{\hat{B}}\otimes\transpose\)
  is bijective, we see that \((\alpha,\beta)\) is a Heisenberg pair if
  and only if \((\tilde\alpha,\tilde\beta)\) is an anti-Heisenberg pair.
\end{proof}

Thus Heisenberg pairs and anti-Heisenberg pairs are essentially
equivalent.

Recall that a bicharacter~\(\bichar\) yields a dual bicharacter
\(\Dubichar\in\U(\hat{B}\otimes\hat{A})\) and a right quantum group
homomorphism \(\Delta_R\colon A\to A\otimes\hat{B}\) by
Theorem~\ref{the:equivalent_notion_of_homomorphisms}.  Similarly,
\(\Dubichar\) yields a right quantum group homomorphism
\(\hat{\Delta}_R\colon B\to B\otimes\hat{A}\).  We reformulate the
condition of being a Heisenberg pair in terms of \(\Dubichar\),
\(\Delta_R\) and~\(\hat{\Delta}_R\), respectively:

\begin{lemma}
  \label{lemm:equiv_cond_V_Heisenberg_pair}
  Let \(\alpha\) and~\(\beta\) be representations of \(A\) and \(B\)
  on a Hilbert space~\(\Hils\).  Then the following are equivalent:
  \begin{enumerate}
  \item \((\alpha,\beta)\) is a Heisenberg pair;
  \item \((\beta,\alpha)\) is a \(\Dubichar\)\nb-Heisenberg pair;
  \item \((\alpha\otimes\Id_{\hat{B}})\Delta_R(a)
    =(\Dumultunit[B]_{\beta 2})(\alpha(a)\otimes 1_{\hat{B}})
    (\Dumultunit[B])^*_{\beta 2}\) for all \(a\in A\);
  \item \((\beta\otimes\Id_{\hat{A}})\hat{\Delta}_R(b) =
    (\Dumultunit[A]_{\alpha 2}) (\beta(b)\otimes 1_{\hat{A}})
    (\Dumultunit[A])^*_{\alpha 2}\) for all \(b\in B\).
  \end{enumerate}
\end{lemma}

\begin{proof}
  (1)\(\iff\)(2): (1) is equivalent to
  \[
  \multunit[A]_{1\alpha}\multunit[B]_{2\beta}\bichar^*_{12}
  =\multunit[B]_{2\beta}\multunit[A]_{1\alpha}
  \qquad \text{in }\U(\hat{A}\otimes\hat{B}\otimes\Comp(\Hils))
  \]
  by~\eqref{eq:V-Heisenberg_pair}.  Applying~\(\flip_{12}\) gives
  \begin{align}
    \label{eq:hat-V-Heisenberg_pair}
    \multunit[A]_{2\alpha}\multunit[B]_{1\beta}\Dubichar_{12}
    &= \multunit[B]_{1\beta}\multunit[A]_{2\alpha}
    \qquad \text{in }\U(\hat{B}\otimes\hat{A}\otimes\Comp(\Hils)),
  \end{align}
  which is equivalent to \((\beta,\alpha)\) being a
  \(\Dubichar\)\nb-Heisenberg pair.  Thus (1)\(\iff\)(2).

  (1)\(\iff\)(3): Let~\((\alpha,\beta)\) be a Heisenberg
  pair.  The following computation takes place in
  \(\U(\hat{A}\otimes\Comp(\Hils)\otimes\hat{B})\):
  \begin{align*}
    (\Id_{\hat{A}}\otimes\alpha\otimes\Id_{\hat{B}})
    (\Id_{\hat{A}}\otimes\Delta_R)\multunit[A]
    &= \multunit[A]_{1\alpha}\bichar_{13}
    = \flip_{23}(\multunit[A]_{1\alpha}\bichar_{12})\\
    &= \flip_{23}((\multunit[B]_{2\beta})^*\multunit[A]_{1\alpha}
    \multunit[B]_{2\beta})
    = (\Dumultunit[B]_{\beta 3})\multunit[A]_{1\alpha}(\Dumultunit[B]_{\beta 3})^*;
  \end{align*}
  the first equality uses~\eqref{eq:def_V_via_right_homomorphism}; the
  second equality is obvious; the third equality
  uses~\eqref{eq:V-Heisenberg_pair}; and the last equality uses
  \(\Dumultunit[B]=\flip((\multunit[B])^*)\).  Since
  \(\{(\omega\otimes\Id_A)\multunit[A]:\omega\in\hat{A}'\}\) is
  linearly dense in~\(A\), slicing the first leg of the first and the
  last expression in the above equation shows that
  (1)\(\Longrightarrow\)(3).

  Conversely, applying
  \(\Id_{\hat{A}}\otimes\alpha\otimes\Id_{\hat{B}}\) on both sides
  of~\eqref{eq:def_V_via_right_homomorphism} and using~(3), we get
  \[
  \multunit[A]_{1\alpha}\bichar_{13}
  = (\Id_{\hat{A}}\otimes(\alpha\otimes\Id_{\hat{B}})\Delta_R)\multunit[A]
  = (\Dumultunit[B]_{\beta 3})\multunit[A]_{1\alpha}(\Dumultunit[B]_{\beta 3})^*
  \quad \text{in }\U(\hat{A}\otimes\Comp(\Hils)\otimes\hat{B});
  \]
  applying~\(\flip_{23}\) to this gives~\eqref{eq:V-Heisenberg_pair}.
  Thus (3)\(\Longrightarrow\)(1).

  To prove (2)\(\iff\)(4) argue as in the proof that (1)\(\iff\)(3).
\end{proof}

\begin{lemma}
  \label{lem:existence_of_canonical_Heisenberg_pair}
  Let \(\HeisPair{\pi}\) and~\(\HeisPair{\eta}\) be \(\G\)\nb- and
  \(\G[H]\)\nb-Heisenberg pairs on Hilbert spaces \(\Hils_\pi\)
  and~\(\Hils_\eta\), respectively.  Then the pair of representations
  \((\alpha,\beta)\) of \(A\) and~\(B\) on
  \(\Hils_\pi\otimes\Hils_\eta\) defined by \(\alpha(a) \defeq
  (\pi\otimes\hat{\eta})\Delta_R(a)\) and \(\beta(b) \defeq
  1_{\Hils_\pi}\otimes\eta(b)\) is a \(\bichar\)\nb-Heisenberg pair.
\end{lemma}

\begin{proof}
  First we check the following equation:
  \begin{equation}
    \label{eq:bichar_char_in_second_leg_in_Heisenberg_pair}
    \bichar_{1\hat\eta}\multunit[B]_{2\eta}
    = \multunit[B]_{2\eta}\bichar_{1\hat\eta}\bichar_{12}\qquad
    \text{in }\U(\hat{A}\otimes\hat{B}\otimes\Comp(\Hils_\eta)).
  \end{equation}
  The coaction \(\hat{B}\to\hat{B}\otimes\hat{B}\) associated to the
  reduced bicharacter~\(\multunit[B]\) is the usual
  comultiplication~\(\DuComult[B]\).  Hence
  \[
  (\Dumultunit[B]_{\eta 3}) \bichar_{1\hat\eta}
  (\Dumultunit[B]_{\eta 3})^*
  = (\Id_{\hat{A}}\otimes\hat\eta\otimes\Id_{\hat{B}}) (\Id\otimes\DuComult[B]) \bichar
  = (\Id_{\hat{A}}\otimes\hat\eta\otimes\Id_{\hat{B}}) (\bichar_{12}\bichar_{13})
  = \bichar_{1\hat\eta}\bichar_{13}
  \]
  in \(\U(\hat{A}\otimes\Comp(\Hils_\eta)\otimes\hat{B})\)
  because of Lemma~\ref{lemm:equiv_cond_V_Heisenberg_pair}(4) and the
  bicharacter property~\eqref{eq:bichar_char_in_second_leg}
  of~\(\bichar\).  When we flip the last two legs, we turn
  \(\Dumultunit[B]_{\eta 3}\) into~\((\multunit[B]_{2\eta})^*\).
  Rearranging then
  gives~\eqref{eq:bichar_char_in_second_leg_in_Heisenberg_pair}.

  Now we can check that \((\alpha,\beta)\) is a Heisenberg pair.  The
  following computation takes place in \(\U(\hat{A}\otimes
  \hat{B}\otimes \Comp(\Hils_\pi)\otimes \Comp(\Hils_\eta))\):
  \[
  \multunit[A]_{1\alpha}\multunit[B]_{2\beta}
  = \multunit[A]_{1\pi}\bichar_{1\hat{\eta}}\multunit[B]_{2\eta}
  = \multunit[A]_{1\pi}\multunit[B]_{2\eta}\bichar_{1\hat{\eta}}\bichar_{12}
  = \multunit[B]_{2\eta}\multunit[A]_{1\pi}\bichar_{1\hat{\eta}}\bichar_{12}
  = \multunit[B]_{2\beta}\multunit[A]_{1\alpha}\bichar_{12}
  \]
  the first equality uses the definitions of \(\alpha\) and~\(\beta\)
  and~\eqref{eq:def_V_via_right_homomorphism}; the second equality
  uses~\eqref{eq:bichar_char_in_second_leg_in_Heisenberg_pair}; the
  third equality uses that \(\multunit[A]_{1\pi}\)
  and~\(\multunit[B]_{2\eta}\) commute; and the fourth equality uses
  the definitions of \(\alpha\) and~\(\beta\) again.
\end{proof}

\subsection{Commutativity and Heisenberg pairs}
\label{sec:commutativity}

Locality principles in quantum field theory always require commutation
relations of the simplest possible form \(xy=yx\).  Our noncommutative
tensor products are based on more complicated commutation relations.
We get ordinary commutativity, however, if we put a Heisenberg and an
anti-Heisenberg pair together.  This is crucial for our noncommutative
tensor product to exist.

\begin{proposition}
  \label{prop:Commutation_V-Heisen_V-anti_Heisen}
  Let \(\Hils\) and~\(\Hils[K]\) be Hilbert spaces; let \(\alpha\)
  and~\(\beta\) be representations of \(A\) and~\(B\) on~\(\Hils\),
  respectively; and let \(\bar\alpha\) and~\(\bar\beta\) be representations
  of \(A\) and~\(B\) on~\(\Hils[K]\), respectively.  Then the
  following are equivalent:
  \begin{enumerate}
  \item the representations \((\alpha\otimes\bar\alpha)\Comult[A]\) and
    \((\beta\otimes\bar\beta)\Comult[B]\) of \(A\) and~\(B\) on
    \(\Hils\otimes\Hils[K]\) commute, that is, for any \(a\in A\) and
    \(b\in B\), we have
    \begin{equation}
      \label{eq:commutator_of_V-Heisen_and_V-anti_Heisen}
      [(\alpha\otimes\bar\alpha)\Comult[A](a),
      (\beta\otimes\bar\beta)\Comult[B](b)]
      =0;
    \end{equation}
  \item there is a bicharacter \(\bichar\in\U(\hat{A}\otimes\hat{B})\)
    such that \((\alpha,\beta)\) is a \(\bichar\)\nb-Heisenberg pair
    and \((\bar\alpha,\bar\beta)\) is a \(\bichar\)\nb-anti-Heisenberg
    pair.
  \end{enumerate}
\end{proposition}

\begin{proof}
  Equation~\eqref{eq:commutator_of_V-Heisen_and_V-anti_Heisen} is
  equivalent to
  \[
  \multunit[A]_{1\alpha}\multunit[A]_{1\bar\alpha}
  \multunit[B]_{2\beta}\multunit[B]_{2\bar\beta}
  = \multunit[B]_{2\beta}\multunit[B]_{2\bar\beta}
  \multunit[A]_{1\alpha}\multunit[A]_{1\bar\alpha}
  \qquad\text{in }\U(\hat{A}\otimes\hat{B}\otimes
  \Comp(\Hils)\otimes\Comp(\Hils[K]))
  \]
  because of \eqref{eq:Comult_W} and~\eqref{eq:first_leg_slice} for
  \(\multunit[A]\) and~\(\multunit[B]\).  Commuting
  \(\multunit[A]_{1\bar\alpha}\) with~\(\multunit[B]_{2\beta}\) and
  \(\multunit[B]_{2\bar\beta}\) with~\(\multunit[A]_{1\alpha}\),
  Equation~\eqref{eq:commutator_of_V-Heisen_and_V-anti_Heisen} becomes
  \begin{equation}
    \label{eq:commutator_equiv_V-H_anti-H_1}
    (\multunit[A])^*_{1\alpha}(\multunit[B])^*_{2\beta}
    \multunit[A]_{1\alpha}\multunit[B]_{2\beta}
    = \multunit[B]_{2\bar\beta}\multunit[A]_{1\bar\alpha}
    (\multunit[B])^*_{2\bar\beta}(\multunit[A])^*_{1\bar\alpha}.
  \end{equation}
  If \(\bichar\in\U(\hat{A}\otimes\hat{B})\) is a bicharacter,
  \((\alpha,\beta)\) a Heisenberg pair and \((\bar\alpha,\bar\beta)\)
  an anti-Heisenberg pair, then
  \eqref{eq:commutator_equiv_V-H_anti-H_1} follows, and
  hence~\eqref{eq:commutator_of_V-Heisen_and_V-anti_Heisen}.

  It remains to show~(2)
  assuming~\eqref{eq:commutator_equiv_V-H_anti-H_1}.
  Let~\(\bichar\) be the unitary
  in~\eqref{eq:commutator_equiv_V-H_anti-H_1}.  It belongs to
  \(1\otimes\hat{A}\otimes\hat{B}\otimes1\) because its first
  definition has~\(1_{\Hils[K]}\) in the fourth leg and its second
  definition has~\(1_{\Hils}\) in the third leg.

  We check that~\(\bichar\) is a bicharacter in both legs.  First we
  check~\eqref{eq:bichar_char_in_first_leg}:
  \begin{align*}
    (\DuComult[A]\otimes\Id_{\hat{B}})\bichar
    &= (\DuComult[A]\otimes\Id_{\hat{B}})
    (\multunit[B]_{2\bar\beta}\multunit[A]_{1\bar\alpha}
    (\multunit[B])^*_{2\bar\beta}(\multunit[A])^*_{1\bar\alpha})\\
    &= \multunit[B]_{3\bar\beta}\multunit[A]_{2\bar\alpha}\multunit[A]_{1\bar\alpha}
    (\multunit[B])^*_{3\bar\beta}(\multunit[A])^*_{1\bar\alpha}(\multunit[A])^*_{2\bar\alpha}\\
    &= \bichar_{23}\multunit[A]_{2\bar\alpha}\bichar_{13}(\multunit[A])^*_{2\bar\alpha}
    =  \bichar_{23}\bichar_{13};
  \end{align*}
  the first and third equality use that~\(\bichar\) is the right hand
  side of~\eqref{eq:commutator_equiv_V-H_anti-H_1}; the second
  equality uses~\eqref{eq:bichar_char_in_first_leg}
  for~\(\multunit[A]\); and the last equality uses that
  \(\multunit[A]_{2\bar\alpha}\) and~\(\bichar_{13}\) commute.  A
  similar computation using that~\(\bichar\) is the left hand side
  in~\eqref{eq:commutator_equiv_V-H_anti-H_1} yields
  \((\Id_{\hat{A}}\otimes\DuComult[B])\bichar=\bichar_{12}\bichar_{13}\);
  thus \(\bichar\in\U(\hat{A}\otimes\hat{B})\) is a bicharacter.
  Finally, \eqref{eq:commutator_equiv_V-H_anti-H_1} says
  that~\((\alpha,\beta)\) is a \(\bichar\)\nb-Heisenberg pair and
  that~\((\bar\alpha,\bar\beta)\) is a \(\bichar\)\nb-anti-Heisenberg
  pair.
\end{proof}

\subsection{Twisted tensor products via Heisenberg pairs}
\label{sec:boxtimes_via_Heisenberg}

Let \(\Qgrp{G}{A}\) and \(\Qgrp{H}{B}\) be \(\Cst\)\nb-quantum
groups, let \(\bichar\in\U(\hat{A}\otimes\hat{B})\) be a
bicharacter, let \((C,\gamma)\) be a
\(\G\)\nb-\(\Cst\)-algebra, and let \((D,\delta)\) be an
\(\G[H]\)\nb-\(\Cst\)-algebra.  Let~\((\alpha,\beta)\) be a
\(\bichar\)\nb-Heisenberg pair on some Hilbert space~\(\Hils\).

Using this data, we now construct a crossed product
\((C\boxtimes_\bichar D,\iota_C,\iota_D)\) of \(C\) and~\(D\)
in the sense of Definition~\ref{def:crossed_product}.  A
more precise notation is
\[
C \boxtimes_\bichar D = (C,\gamma)\boxtimes_\bichar (D,\delta).
\]
There is no need to mention~\((\alpha,\beta)\) in our notation
because all Heisenberg pairs give equivalent crossed products; we
will prove this in Section~\ref{sec:boxtimes_with_covariant_rep}.
Our definition is based on the morphisms
\begin{alignat*}{2}
  \iota_C&\colon C\to C\otimes D\otimes \Comp(\Hils),&\qquad
  c&\mapsto (\Id_C\otimes\alpha)\gamma(c)_{13},\\
  \iota_D&\colon D\to C\otimes D\otimes \Comp(\Hils),&\qquad
  d&\mapsto (\Id_D\otimes\beta)\delta(d)_{23}.
\end{alignat*}

\begin{lemma}
  \label{lem:boxtimes_crossed_product}
  Let \(X\subseteq C\) and \(Y\subseteq D\) be closed subspaces
  with
  \[
  \gamma(X)\cdot (1_C\otimes A) = X\otimes A
  \quad\text{and}\quad
  \delta(Y)\cdot (1_D\otimes B) = Y\otimes B.
  \]
  Then \(\iota_C(X)\cdot \iota_D(Y) = \iota_D(Y)\cdot
  \iota_C(X)\) in \(\Mult(C\otimes D\otimes\Comp(\Hils))\).
\end{lemma}

\begin{proof}
  Let \((\bar\alpha,\bar\beta)\) be a \(\bichar\)\nb-anti-Heisenberg
  pair on a Hilbert space~\(\Hils[K]\).  The definition of~\(\iota_C\)
  and the comodule property~\eqref{eq:right_coaction} for~\(\gamma\)
  yield
  \[
  (\iota_C\otimes\bar\alpha)\gamma
  = ((\Id_C\otimes\alpha\otimes\bar\alpha)
  (\gamma\otimes\Id_A)\gamma)_{134}
  = ((\Id_C\otimes(\alpha\otimes\bar\alpha)
  \Comult[A])\gamma)_{134};
  \]
  Similarly,
  \[
  (\iota_D\otimes\bar\beta)\delta
  = ((\Id_D\otimes(\beta\otimes\bar\beta)\Comult[B])\delta)_{234}.
  \]
  Now Proposition~\ref{prop:Commutation_V-Heisen_V-anti_Heisen}
  yields
  \begin{equation}
    \label{eq:commutation_of_C_and_D_morphism}
    (\iota_C\otimes\bar\alpha)\gamma(c)\cdot
    (\iota_D\otimes\bar\beta)\delta(d)
    = (\iota_D\otimes\bar\beta)\delta(d) \cdot
    (\iota_C\otimes\bar\alpha)\gamma(c)
  \end{equation}
  for all \(c\in C\), \(d\in D\).

  Since \(\bar\alpha(A)\cdot\Comp(\Hils[K]) =
  \Comp(\Hils[K])\), our assumption \(\gamma(X)\cdot
  (1_C\otimes A) = X\otimes A\) gives
  \begin{multline*}
    ((\iota_C\otimes\bar\alpha)\gamma(X))\cdot
    \Comp(\Hils[K])_4
    = (\iota_C\otimes\bar\alpha)
    (\gamma(X)\cdot(1_C\otimes A)) \cdot
    \Comp(\Hils[K])_4
    \\ = (\iota_C(X)\otimes\bar\alpha(A)) \cdot
    \Comp(\Hils[K])_4
    = \iota_C(X)\otimes\Comp(\Hils[K]).
  \end{multline*}
  Similarly, \(\bar\beta(B)\cdot\Comp(\Hils[K])=\Comp(\Hils[K])\) and
  \(\delta(Y)\cdot (1_D\otimes B) = Y\otimes B\) give
  \[
  ((\iota_D\otimes\bar\beta)\delta(Y))\cdot \Comp(\Hils[K])_4
  = \iota_D(Y)\otimes\Comp(\Hils[K]).
  \]
  Equation~\eqref{eq:commutation_of_C_and_D_morphism} gives
  \[
  (\iota_C\otimes\bar\alpha)\gamma(X)\cdot
  (\iota_D\otimes\bar\beta)\delta(Y) =
  (\iota_D\otimes\bar\beta)\delta(Y) \cdot
  (\iota_C\otimes\bar\alpha)\gamma(X).
  \]
  Multiplying this equation on the right with \(1_{C\otimes
    D\otimes\Hils}\otimes\Comp(\Hils[K])\) and using the
  computations above to simplify, we get
  \[
  (\iota_C(X)\cdot\iota_D(Y))\otimes\Comp(\Hils[K])
  = (\iota_D(X)\cdot\iota_C(Y))\otimes\Comp(\Hils[K]).
  \]
  Applying a state~\(\omega\) on~\(\Comp(\Hils[K])\) to this
  equation gives \(\iota_C(X)\cdot\iota_D(Y) =
  \iota_D(X)\cdot\iota_C(Y)\) as desired.
\end{proof}

\begin{lemma}
  \label{lem:boxtimes_Cstar}
  \(\iota_C(C)\cdot \iota_D(D) = \iota_D(D)\cdot \iota_C(C)\) in
  \(\Mult(C\otimes D\otimes\Comp(\Hils))\).
\end{lemma}

\begin{proof}
  Since our coactions satisfy the Podle\'s conditions, this is the
  special case \(X=C\) and \(Y=D\) of
  Lemma~\ref{lem:boxtimes_crossed_product}.
\end{proof}

Lemma~\ref{lem:boxtimes_Cstar} and the discussion after
Definition~\ref{def:crossed_product} imply that
\[
C\boxtimes_\bichar D \defeq \iota_C(C)\cdot\iota_D(D)
\]
is a \(\Cst\)\nb-algebra, that \(\iota_C\) and~\(\iota_D\) are
morphisms from \(C\) and~\(D\) to \(C\boxtimes_\bichar D\),
respectively, and that \((C\boxtimes_\bichar D,\iota_C,\iota_D)\) is
a crossed product of \(C\) and~\(D\).

The following observation is useful to study slice maps on
\(C\boxtimes_\bichar D\).

\begin{lemma}
  \label{lem:boxtimes_Podles}
  In the situation of
  Lemma~\textup{\ref{lem:boxtimes_crossed_product}},
  \begin{equation}
    \label{eq:like_Podles_for_boxtimes}
    \iota_C(X)\cdot \iota_D(Y)\cdot
    \Comp(\Hils)_3
    = X \otimes Y \otimes \Comp(\Hils),
  \end{equation}
  where the right hand side means the closed linear span of
  \(x\otimes y\otimes z\) with \(x\in X\), \(y\in Y\),
  \(z\in\Comp(\Hils)\).  In particular, \((C\boxtimes_\bichar
  D) \cdot \Comp(\Hils)_3 = C\otimes D\otimes \Comp(\Hils)\).
\end{lemma}

\begin{proof}
  Since \(\Comp(\Hils)= \beta(B)\cdot\Comp(\Hils)\), we may
  compute
  \begin{multline*}
    \iota_D(Y)\cdot \Comp(\Hils)_3
    = ((\Id_D\otimes\beta)(\delta(Y)\cdot (1_D\otimes B)))_{23}
    \Comp(\Hils)_3
    \\ = (Y \otimes \beta(B)\cdot \Comp(\Hils))_{23}
    = \Comp(\Hils)_3 \cdot Y_2.
  \end{multline*}
  Here \(Y_2\) and \(\Comp(\Hils)_3\) mean \(Y\) and
  \(\Comp(\Hils)\) in the second and third leg, respectively.
  A similar computation for \(\iota_C(X)\) using
  \(\Comp(\Hils)= \alpha(A)\cdot\Comp(\Hils)\) now
  gives~\eqref{eq:like_Podles_for_boxtimes}.
\end{proof}

\begin{example}
  \label{exa:boxtimes_trivial_coaction}
  Assume that the coaction~\(\gamma\) is trivial.  Then
  \(\gamma(c)_{13}=c\otimes 1\otimes 1\), so that \(C\boxtimes_\bichar
  D\cong C\otimes D\), embedded into \(\Mult(C\otimes
  D\otimes\Comp(\Hils))\) via \(\Id_C\otimes
  (\Id_D\otimes\beta)\delta\).  We also get \(C\boxtimes_\bichar
  D\cong C\otimes D\) if~\(\delta\) is trivial.
\end{example}

\begin{lemma}
  \label{lemm:crossed_associativity}
  Let \(C_0\) and~\(D_0\) be \(\Cst\)\nb-algebras and equip
  \(C_0\otimes C\) and \(D_0\otimes D\) with the coactions
  \(\Id_{C_0}\otimes\gamma\) and \(\Id_{D_0}\otimes\delta\),
  respectively.  Then
  \begin{equation}
    \label{eq:left_asso}
    (C_0\otimes C)\boxtimes_\bichar (D_0 \otimes D)
    = C_0\otimes D_0 \otimes (C\boxtimes_\bichar D).
  \end{equation}
\end{lemma}

\begin{proof}
  The maps \(\iota_{C_0\otimes C}\) and \(\iota_{D_0\otimes D}\) are
  \(\Id_{C_0}\otimes \iota_C\) and \(\Id_{D_0}\otimes \iota_D\),
  respectively.
\end{proof}

\section{Hilbert space representation of the twisted tensor product}
\label{sec:boxtimes_with_covariant_rep}

Let \(\Qgrp{G}{A}\), \(\Qgrp{H}{B}\),
\(\bichar\in\U(\hat{A}\otimes\hat{B})\), \((C,\gamma)\) and
\((D,\delta)\) be as before, so that the twisted tensor product
\(C\boxtimes_\bichar D\) is defined.  We are going to construct a
faithful Hilbert space representation of \(C\boxtimes_\bichar D\)
using covariant Hilbert space representations of \((C,\gamma)\) and
\((D,\delta)\).  This yields an alternative definition
of~\(C\boxtimes_\bichar D\) and shows that \(C\boxtimes_\bichar
D\) does not depend on the Heisenberg pair used
in its construction.

Our new construction uses faithful covariant representations
\((\varphi,U^{\Hils})\) of \((C,\gamma,A)\) and
\((\psi,U^{\Hils[K]})\) of \((D,\delta,B)\) on Hilbert spaces
\(\Hils\) and~\(\Hils[K]\), respectively.  (We will show below that
such faithful covariant representations always exist.)

The bicharacter~\(\bichar\) and the corepresentations provide a
unitary operator~\(Z\) on \(\Hils\otimes\Hils[K]\) as follows:

\begin{theorem}
  \label{the:V-Heis_Comm_corep}
  Let \(U^{\Hils}\in\U(\Hils\otimes A)\) and
  \(U^{\Hils[K]}\in\U(\Hils[K]\otimes B)\) be corepresentations
  of \(\G\) and~\(\G[H]\), respectively.  Then there is a unique
  unitary \(Z\in\U(\Hils\otimes\Hils[K])\) that satisfies
  \begin{equation}
    \label{eq:commutation_between_corep_of_two_quantum_groups}
    U^{\Hils}_{1\alpha} U^{\Hils[K]}_{2\beta} Z_{12}
    = U^{\Hils[K]}_{2\beta} U^{\Hils}_{1\alpha}
    \qquad\text{in }\U(\Hils\otimes\Hils[K]\otimes\Hils[L])
  \end{equation}
  for any \(\bichar\)\nb-Heisenberg pair \((\alpha,\beta)\) on any
  Hilbert space~\(\Hils[L]\).
\end{theorem}

With this unitary~\(Z\), define representations \(\varphi_1\)
and~\(\tilde\psi_2\) of \(C\) and~\(D\) on \(\Hils\otimes\Hils[K]\) by
\begin{align*}
  \varphi_1(c) &\defeq \varphi(c)\otimes 1_{\Hils[K]},\\
  \tilde\psi_2(d) &\defeq Z (1_{\Hils}\otimes \psi(d))Z^*.
\end{align*}

\begin{theorem}
  \label{the:crossed_prod_for_cov_reps_in_Hilb_sp_lavel}
  Let \((\varphi,U^{\Hils})\) and \((\psi,U^{\Hils[K]})\) be faithful
  covariant representations of \((C,\gamma,A)\) and \((D,\delta,B)\)
  on Hilbert spaces \(\Hils\) and~\(\Hils[K]\), respectively.
  Construct \(\varphi_1\) and~\(\tilde\psi_2\) as above.  Then there
  is a unique faithful representation \(\rho\colon C\boxtimes_\bichar
  D\to\Bound(\Hils\otimes\Hils[K])\) with \(\rho\circ\iota_C =
  \varphi_1\) and \(\rho\circ\iota_D = \tilde\psi_2\).
\end{theorem}

\begin{example}
  \label{exa:boxtimes_trivial_bicharacter}
  If \(\bichar=1\), then we may take \(Z=1\).  Thus
  \(\tilde\psi_2=\psi_2\) and the crossed product is simply the
  minimal tensor product \(C\otimes D\).
\end{example}

In the rest of this section, we prove the claims above and use
the main theorem to show that the twisted tensor product does
not depend on auxiliary choices.

First we construct faithful covariant representations:

\begin{example}
  \label{exa:existence_of_faithful_covariant_representation}
  Let \(\varphi_0\colon C\to\Bound(\Hils_0)\) be any faithful
  Hilbert space representation.  Let \(\HeisPair{\pi}\) be a
  faithful \(\G\)\nb-Heisenberg pair on a Hilbert
  space~\(\Hils_\pi\); this exists because of
  Example~\ref{exa:Heisenberg_from_Multunit}.  Let
  \(\Hils\defeq \Hils_0\otimes\Hils_\pi\) and identify
  \(\Comp(\Hils)\cong \Comp(\Hils_0)\otimes\Comp(\Hils_\pi)\).
  The unitary \(U\defeq 1_{\Hils_0}\otimes
  \multunit[A]_{\hat{\pi}2}\in \U(\Comp(\Hils)\otimes A)\)
  is a corepresentation; since \(\varphi_0\), \(\pi\)
  and~\(\gamma\) are faithful morphisms, \(\varphi\defeq
  (\varphi_0\otimes\pi)\circ\gamma\colon C\to \Bound(\Hils)\)
  is a faithful representation.  The following computation in
  \(\Mult(C\otimes\Comp(\Hils_\pi)\otimes A)\) implies the
  covariance condition for \((\varphi,U)\):
  \[
  ((\Id_C\otimes\pi)\gamma\otimes\Id_A)\gamma(c)
  = (\Id_C\otimes(\pi\otimes\Id_A)\Comult[A])\gamma(c)
  = (\multunit[A]_{\hat{\pi}3}) (\gamma(c)\otimes 1_A)
  (\multunit[A]_{\hat{\pi}3})^*
  \]
  for all \(c\in C\), where we used \eqref{eq:right_coaction} and
  Lemma~\ref{lemm:equiv_cond_V_Heisenberg_pair}.3 with \(B=\hat{A}\)
  and \(\Delta_R=\Comult[A]\).
\end{example}

Now we prove Theorem~\ref{the:V-Heis_Comm_corep}.  The uniqueness
of~\(Z\) is clear from
\[
Z_{12} = (U^{\Hils[K]}_{2\beta})^*
(U^{\Hils}_{1\alpha})^* U^{\Hils[K]}_{2\beta} U^{\Hils}_{1\alpha}.
\]
Existence means that the operator on the right acts identically on
the third leg and does not depend on the Heisenberg pair.  The
quickest way to prove this uses universal quantum groups to turn
corepresentations into representations.  Let
\(\rho_1\colon\hat{A}^\univ\to\Bound(\Hils)\) and \(\rho_2\colon
\hat{B}^\univ\to\Bound(\Hils[K])\) be the unique representations
with \((\rho_1\otimes\Id_A)\maxcorep[A]=U^{\Hils}\) and
\((\rho_2\otimes\Id_B)\maxcorep[B]=U^{\Hils[K]}\) (see
Section~\ref{sec:univ_qgr}).

The bicharacter \(\bichar\in\U(\hat{A}\otimes\hat{B})\) lifts uniquely
to a bicharacter \(\bichar^\univ\in
\U(\hat{A}^\univ\otimes\hat{B}^\univ)\) by
Theorem~\ref{the:equivalent_notion_of_homomorphisms}.  We claim that
\begin{equation}
  \label{eq:Z_via_universal_lift}
  Z\defeq (\rho_1\otimes\rho_2)(\bichar^\univ)^*
  \in\Bound(\Hils\otimes\Hils[K])
\end{equation}
verifies~\eqref{eq:commutation_between_corep_of_two_quantum_groups}
(for any \(\bichar\)\nb-Heisenberg pair~\((\alpha,\beta)\)).
(Our formulation of Theorem~\ref{the:V-Heis_Comm_corep}
highlights the property of the operator
\((\rho_1\otimes\rho_2)(\bichar^\univ)^*\) that is crucial for
the proof of
Theorem~\ref{the:crossed_prod_for_cov_reps_in_Hilb_sp_lavel},
and it avoids universal quantum groups.)

We will actually prove
\begin{equation}
  \label{eq:V-Heis_pair_in_univ_level}
  \maxcorep[A]_{1\alpha}\maxcorep[B]_{2\beta}
  = \maxcorep[B]_{2\beta}\maxcorep[A]_{1\alpha}\bichar^\univ_{12}
  \qquad\text{in }\U(\hat{A}^\univ\otimes \hat{B}^\univ\otimes
  \Comp(\Hils[L]))
\end{equation}
for any \(\bichar\)\nb-Heisenberg pair \((\alpha,\beta)\).  Applying
\(\rho_1\) and~\(\rho_2\) to the first two legs then
gives~\eqref{eq:commutation_between_corep_of_two_quantum_groups}
because \((\rho_1\otimes\Id_A)\maxcorep[A]=U^{\Hils}\) and
\((\rho_2\otimes\Id_B)\maxcorep[B]=U^{\Hils[K]}\).

When we apply the reducing morphisms \(\Lambda_{\hat{A}}\colon
\hat{A}^\univ\to \hat{A}\) and \(\Lambda_{\hat{B}}\colon
\hat{B}^\univ\to \hat{B}\) to the first two legs
in~\eqref{eq:V-Heis_pair_in_univ_level}, we get
\(\multunit[A]_{1\alpha} \multunit[B]_{2\beta} = \multunit[B]_{2\beta}
\multunit[A]_{1\alpha} \bichar_{12}\), which is exactly the definition
of a Heisenberg pair (see
Definition~\ref{def:V-Heisenberg_pair}).  A routine computation shows
that
\[
T\defeq (\maxcorep[A]_{1\alpha})^*
(\maxcorep[B]_{2\beta})^* \maxcorep[A]_{1\alpha}
\maxcorep[B]_{2\beta}\in \U(\hat{A}^\univ\otimes
\hat{B}^\univ\otimes \Comp(\Hils[L]))
\]
is a character in the first two legs, that is,
\((\DuComult[A^\univ]\otimes\Id_{\hat{B}^\univ}\otimes
\Id_{\Hils[L]})T=T_{234}T_{134}\) and
\((\Id_{\hat{A}^\univ}\otimes\DuComult[B^\univ]\otimes
\Id_{\Hils[L]})T=T_{124}T_{134}\).
Thus \(T\) and~\(\bichar^\univ_{12}\) are two bicharacters in
\(\U(\hat{A}^\univ\otimes \hat{B}^\univ\otimes \Comp(\Hils[L]))\)
that both lift the bicharacter~\(\bichar_{12}\) in
\(\U(\hat{A}\otimes \hat{B}\otimes \Comp(\Hils[L]))\).  Using
\cite{Meyer-Roy-Woronowicz:Homomorphisms}*{Lemma 4.6} twice, we get
that any such bicharacter has a unique lifting.  Thus
\(T=\bichar^\univ_{12}\) as asserted.  This finishes the proof of
Theorem~\ref{the:V-Heis_Comm_corep}.

Now we come to the proof of
Theorem~\ref{the:crossed_prod_for_cov_reps_in_Hilb_sp_lavel}.  The
Hilbert space representation
\[
\varphi\otimes\psi\otimes\Id\colon
C\otimes D\otimes \Comp(\Hils[L]) \to
\Bound(\Hils\otimes \Hils[K] \otimes \Hils[L])
\]
is faithful because \(\varphi\) and~\(\psi\) are faithful.  Hence the
pair of representations
\begin{align*}
  (\varphi\otimes\alpha)\gamma_{13}&\colon C\to
  \Bound(\Hils\otimes \Hils[K] \otimes \Hils[L])\\
  (\psi\otimes\beta)\delta_{23}&\colon D\to
  \Bound(\Hils\otimes \Hils[K] \otimes \Hils[L])
\end{align*}
of \(C\) and~\(D\) gives a faithful representation of
\(C\boxtimes_\bichar D\); that is, there is a unique faithful representation of \(C\boxtimes_\bichar D\) that gives the above two
representations when composed with \(\iota_C\) and~\(\iota_D\).

\begin{lemma}
  \label{lem:conjugate_rep_boxtimes}
  The pair of representations \((\varphi_1,
  \Ad_{Z_{12}}\circ\psi_2)\) of \((C,D)\) on \(\Hils\otimes
  \Hils[K] \otimes \Hils[L]\) is unitarily equivalent to the
  pair \(\bigl((\varphi\otimes\alpha)\gamma_{13},
  (\psi\otimes\beta)\delta_{23}\bigr)\) on the same Hilbert
  space through conjugation by the unitary
  \(U^{\Hils}_{1\alpha} U^{\Hils[K]}_{2\beta}\).
\end{lemma}

\begin{proof}
  We must prove
  \begin{align*}
    U^{\Hils}_{1\alpha} U^{\Hils[K]}_{2\beta} (\varphi(c)\otimes
    1_{\Hils[K]}\otimes 1_{\Hils[L]}) (U^{\Hils[K]}_{2\beta})^*
    (U^{\Hils}_{1\alpha})^*
    &= (\varphi\otimes\alpha)\gamma_{13}(c),\\
    U^{\Hils}_{1\alpha} U^{\Hils[K]}_{2\beta} Z_{12} (1_{\Hils}\otimes
    \psi(d)\otimes 1_{\Hils[L]}) Z_{12}^* (U^{\Hils[K]}_{2\beta})^*
    (U^{\Hils}_{1\alpha})^*
    &= (\psi\otimes\beta)\delta_{23}(d)
  \end{align*}
  for all \(c\in C\), \(d\in D\).  To check the first equality,
  we use first that~\(U^{\Hils[K]}_{2\beta}\) commutes
  with~\(\varphi(c)_1\) because both act on different legs, and
  secondly the covariance condition~\eqref{eq:covariant_corep} for
  \((\varphi,U^{\Hils})\) with~\(\alpha\) applied to the leg~\(A\):
  \[
  U^{\Hils}_{1\alpha} U^{\Hils[K]}_{2\beta} (\varphi(c)\otimes
  1_{\Hils[K]}\otimes 1_{\Hils[L]}) (U^{\Hils[K]}_{2\beta})^*
  (U^{\Hils}_{1\alpha})^*
  = U^{\Hils}_{1\alpha} (\varphi(c)\otimes 1_{\Hils[K]}\otimes
  1_{\Hils[L]}) (U^{\Hils}_{1\alpha})^*
  = (\varphi\otimes\alpha)\gamma(c)_{13}.
  \]
  A similar computation gives the second equality:
  \begin{multline*}
    U^{\Hils}_{1\alpha} U^{\Hils[K]}_{2\beta} Z_{12} (1_{\Hils}\otimes
    \psi(d)\otimes 1_{\Hils[L]}) Z_{12}^* (U^{\Hils[K]}_{2\beta})^*
    (U^{\Hils}_{1\alpha})^*
    = U^{\Hils[K]}_{2\beta} U^{\Hils}_{1\alpha} \psi(d)_2
    (U^{\Hils}_{1\alpha})^* (U^{\Hils[K]}_{2\beta})^*
    \\= U^{\Hils[K]}_{2\beta} \psi(d)_2 (U^{\Hils[K]}_{2\beta})^*
    = (\psi\otimes\beta)\delta(d)_{23};
  \end{multline*}
  we first
  use~\eqref{eq:commutation_between_corep_of_two_quantum_groups};
  secondly, that \(U^{\Hils}_{1\alpha}\) and~\(\psi(d)_2\) act in
  different legs to commute them; and thirdly the covariance
  condition~\eqref{eq:covariant_corep} for \((\psi,U^{\Hils[K]})\)
  with~\(\beta\) applied to the leg~\(B\).
\end{proof}

We remarked above that the pair of representations
\(\bigl((\varphi\otimes\alpha)\gamma_{13},
(\psi\otimes\beta)\delta_{23}\bigr)\) generates a faithful
representation of~\(C\boxtimes_\bichar D\).
Lemma~\ref{lem:conjugate_rep_boxtimes} shows that this
representation is unitarily equivalent to another
representation that restricts to \(\varphi_1\otimes
1_{\Hils[L]}\) and
\(\Ad_{Z_{12}}\circ\psi_2=\tilde\psi_2\otimes 1_{\Hils[L]}\) on
\(C\) and~\(D\), respectively.  The latter representation is
\(\rho\otimes 1_{\Hils[L]}\) for a faithful representation
of~\(C\boxtimes_\bichar D\) on \(\Hils\otimes\Hils[K]\).  This
is the faithful representation whose existence is asserted in
Theorem~\ref{the:crossed_prod_for_cov_reps_in_Hilb_sp_lavel}.
Uniqueness is clear because \(C\boxtimes_\bichar D =
\iota_C(C)\cdot \iota_D(D)\).  This finishes the proof of
Theorem~\ref{the:crossed_prod_for_cov_reps_in_Hilb_sp_lavel}.

\medskip

\begin{theorem}
  \label{lem:boxtimes_well-defined}
  In the notation of
  Theorem~\textup{\ref{the:crossed_prod_for_cov_reps_in_Hilb_sp_lavel}},
  the subspace
  \[
  C\hboxt_\bichar D \defeq \varphi_1(C)\cdot \tilde\psi_2(D)
  \subseteq \Bound(\Hils\otimes\Hils[K])
  \]
  is a \(\Cst\)\nb-subalgebra and \((C\hboxt_\bichar
  D,\varphi_1,\tilde\psi_2)\) is a crossed product of \(C\)
  and~\(D\).  Up to equivalence of crossed products, it does
  not depend on \((\varphi,U^{\Hils})\) and
  \((\psi,U^{\Hils[K]})\).

  The crossed product \((C\boxtimes_\bichar
  D,\iota_C,\iota_D)\) is equivalent to \((C\hboxt_\bichar
  D,\varphi_1,\tilde\psi_2)\) and, up to equivalence of crossed
  products, does not depend on the Heisenberg
  pair~\((\alpha,\beta)\).
\end{theorem}

\begin{proof}
  Since \(C\hboxt_\bichar D=\rho(C\boxtimes_\bichar D)\) and
  \(\rho\circ\iota_C=\varphi_1\),
  \(\rho\circ\iota_D=\tilde\psi_2\), by
  Theorem~\ref{the:crossed_prod_for_cov_reps_in_Hilb_sp_lavel},
  \(C\hboxt_\bichar D\) is a \(\Cst\)\nb-algebra,
  \((C\hboxt_\bichar D,\varphi_1,\tilde\psi_2)\) is a crossed
  product of \(C\) and~\(D\), and it is equivalent to the
  crossed product \((C\boxtimes_\bichar D,\iota_C,\iota_D)\).

  Since the unitary~\(Z\) is the same for all Heisenberg pairs
  \((\alpha,\beta)\), the crossed product \((C\hboxt_\bichar
  D,\varphi_1,\tilde\psi_2)\) does not depend on
  \((\alpha,\beta)\); hence up to equivalence
  \((C\boxtimes_\bichar D,\iota_C,\iota_D)\) does not depend on
  \((\alpha,\beta)\).  And since \((C\boxtimes_\bichar
  D,\iota_C,\iota_D)\) does not depend on
  \((\varphi,U^{\Hils})\) and \((\psi,U^{\Hils[K]})\), neither
  does \((C\hboxt_\bichar D,\varphi_1,\tilde\psi_2)\), up to
  equivalence.
\end{proof}

As a special case of
Theorem~\ref{the:crossed_prod_for_cov_reps_in_Hilb_sp_lavel}, the
usual spatial tensor product \(C\otimes D\) does not depend on the
chosen faithful representations of \(C\) and~\(D\).  But we have not
reproved this classical result.  Rather, we have reduced analogous
statements for noncommutative tensor products to this case by
embedding the latter into commutative tensor products with more
factors.

\section{Properties of the twisted tensor product}
\label{sec:properties_boxtimes}

In this section, we establish several functoriality properties
of the twisted tensor product.  We also discuss exactness for
equivariantly semi-split extensions and invariance under
Morita--Rieffel equivalence, which gives a result about cocycle
conjugacy.

We begin with an easy symmetry property:

\begin{proposition}
  \label{pro:symmetry_of_twisted_tens_prod}
  The crossed products \((C\boxtimes_{\bichar}D,\iota_C,\iota_D)\) and
  \((D\boxtimes_{\Dubichar}C,\iota'_D,\iota'_C)\) are canonically
  isomorphic.
\end{proposition}

\begin{proof}
  Let \((U^{\Hils},\varphi)\) and \((U^{\Hils[K]},\psi)\) be faithful
  covariant representations of \(C\) and~\(D\) on Hilbert spaces
  \(\Hils\) and~\(\Hils[K]\), respectively.  
  Theorem~\ref{lem:boxtimes_well-defined} yields
  \begin{align*}
    (C\boxtimes_{\bichar}D,\iota_C,\iota_C)
    &\cong (C\widetilde{\boxtimes}_{\bichar}D, \varphi_1, \tilde{\psi}_2),\\
    (D\boxtimes_{\Dubichar}C,\iota'_D,\iota'_C)
    &\cong (D\widetilde{\boxtimes}_{\Dubichar}C, \psi_1, \tilde{\varphi}_2)
  \end{align*}
  with \(C\widetilde{\boxtimes}_{\bichar}D\subseteq
  \Bound(\Hils\otimes\Hils[K])\),
  \(D\widetilde{\boxtimes}_{\Dubichar}C\subseteq
  \Bound(\Hils[K]\otimes\Hils)\); here \(\psi_1(d) \defeq
  (\psi(d)\otimes 1_{\Hils})\) and \(\tilde{\varphi}_2(c) \defeq
  \hat{Z}(1_{\Hils[K]}\otimes\psi(c))\hat{Z}^{*}\), where~\(Z\)
  satisfies~\eqref{eq:commutation_between_corep_of_two_quantum_groups}
  and \(\hat{Z}= \Flip Z^* \Flip\).  The pair of representations
  \((\varphi_1,\tilde{\psi}_2)\) of~\((C,D)\)
  on~\(\Hils\otimes\Hils[K]\) is unitarily equivalent to the pair of
  representations \((\tilde{\varphi}_2,\psi_1)\)
  on~\(\Hils[K]\otimes\Hils\) via the unitary~\(\Flip Z^*\).
\end{proof}

\subsection{Functoriality for quantum group morphisms}
\label{sec:functoriality_qgr_morphisms}

Let \(\Qgrp{G}{A}\), \(\Qgrp{H}{B}\), \(\Qgrp{G_2}{A_2}\) and
\(\Qgrp{H_2}{B_2}\) be quantum groups.  Let \(f\colon
\G\to\G_2\) and \(g\colon \G[H]\to\G[H]_2\) be quantum group
morphisms in the sense of the several equivalent descriptions
in Theorem~\ref{the:equivalent_notion_of_homomorphisms}.

Let \(\bichar_2\in\U(\hat{A}_2\otimes\hat{B}_2)\) be a
bicharacter.  We may view~\(\bichar_2\) as a quantum group
morphism \(\bichar_2'\colon \G_2\to\DuG[H]_2\).  Composing this
with the given quantum group morphisms \(f\colon \G\to\G_2\)
and the dual \(\hat{g}\colon \DuG[H]_2\to\DuG[H]\), we get a
quantum group morphism \(\bichar'\defeq \hat{g}\circ
\bichar_2'\circ f\colon \G\to\DuG[H]\), which we view as a
bicharacter \(\bichar\in\U(\hat{A}\otimes\hat{B})\).

Let \((C,\gamma)\) and~\((D,\delta)\) be a \(\G\)\nb-\(\Cst\)-algebra
and an \(\G[H]\)\nb-\(\Cst\)-algebra, respectively.  The description
of~\(f\) in Theorem~\ref{the:equivalent_notion_of_homomorphisms}(4) is
as a functor between the categories of \(\G\)- and
\(\G_2\)\nb-\(\Cst\)-algebras that does not change the underlying
\(\Cst\)\nb-algebras.  In particular, this functor maps~\(\gamma\) to
a continuous \(\G_2\)\nb-coaction \(\gamma_2\colon C\to C\otimes A_2\)
on~\(C\).  Similarly, \(g\) maps~\(\delta\) to a continuous
\(\G[H]_2\)\nb-coaction \(\delta_2\colon D\to D\otimes B_2\) on~\(D\).

\begin{theorem}
  \label{the:functoriality_qgr_morphism}
  In the situation above, the crossed products \((C,\gamma_2)
  \boxtimes_{\bichar_2} (D,\gamma_2)\) and \((C,\gamma)
  \boxtimes_{\bichar} (D,\gamma)\) of \(C\) and~\(D\) are
  equivalent.
\end{theorem}

\begin{proof}
  Let \((\varphi,U^{\Hils})\) be a \(\G\)\nb-covariant
  representation of \((C,\gamma)\) on~\(\Hils\) and let
  \((\psi,U^{\Hils[K]})\) be an \(\G[H]\)\nb-covariant
  representation of \((D,\delta)\) on~\(\Hils[K]\).

  The quantum group morphism~\(f\) turns~\(U^{\Hils}\) into a
  corepresentation~\(U_2^{\Hils}\) of~\(\G_2\) on~\(\Hils\).  This is
  asserted in \cite{Meyer-Roy-Woronowicz:Homomorphisms}*{Proposition
    6.5}.  Since the quick proof given
  in~\cite{Meyer-Roy-Woronowicz:Homomorphisms} only works for
  corepresentations that induce a continuous coaction on
  \(\Comp(\Hils)\), which is not automatic, we give a different proof
  here using universal quantum groups.

  We may view the quantum group morphism~\(f\) as a Hopf \Star{}homomorphism
  \(\hat{f}\colon \hat{A}_2^\univ\to\hat{A}^\univ\) between the
  duals of the associated universal \(\Cst\)\nb-algebras by
  Theorem~\ref{the:equivalent_notion_of_homomorphisms}.  By
  the universal property, \(U^{\Hils}\) is equivalent to a
  representation of~\(\hat{A}^\univ\) on~\(\Hils\).  Composing
  this with~\(\hat{f}\) gives a representation
  of~\(\hat{A}_2^\univ\), which is equivalent to the desired
  corepresentation~\(U_2^{\Hils}\) of~\(\G_2\) on~\(\Hils\).

  This operation on the level of corepresentations is
  compatible with the map \(\gamma\mapsto\gamma_2\) on
  coactions in the sense that \((\varphi,U_2^{\Hils})\) is a
  \(\G_2\)\nb-covariant representation of \((C,\gamma_2)\).

  Similarly, \(g\) turns~\(U^{\Hils[K]}\) into a
  corepresentation~\(U_2^{\Hils[K]}\) of~\(\G[H]_2\)
  on~\(\Hils[K]\), and \((\psi,U_2^{\Hils[K]})\) is a covariant
  representation of \((D,\delta_2)\).

  The bicharacters \(\bichar\in \U(\hat{A} \otimes
  \hat{B})\) and \(\bichar_2\in \U(\hat{A}_2 \otimes
  \hat{B}_2)\) lift uniquely to bicharacters
  \(\bichar^\univ\in \U(\hat{A}^\univ \otimes
  \hat{B}^\univ)\) and \(\bichar_2^\univ\in
  \U(\hat{A}_2^\univ \otimes \hat{B}_2^\univ)\) by
  Theorem~\ref{the:equivalent_notion_of_homomorphisms}.
  The bijection between bicharacters and quantum group
  morphisms is defined in such a way that \(\bichar^\univ =
  (\hat{f}\otimes\hat{g}) (\bichar_2^\univ)\).
  Equation~\ref{eq:Z_via_universal_lift} then shows that the
  unitaries~\(Z\) on \(\Hils\otimes\Hils[K]\) that are used to
  construct the twisted tensor products with respect to
  \(\bichar\) and~\(\bichar_2\) are the same.

  Now
  Theorem~\ref{the:crossed_prod_for_cov_reps_in_Hilb_sp_lavel}
  yields the desired equivalence of crossed products because
  both are faithfully represented by the same
  \(\Cst\)\nb-algebra \(\varphi(C)\cdot Z\psi(D)Z^*\) on
  \(\Hils\otimes\Hils[K]\).
\end{proof}

The following special cases of
Theorem~\ref{the:functoriality_qgr_morphism} are particularly
noteworthy.

\begin{example}
  \label{exa:reduce_to_B}
  Let \(\G_2=\DuG[H]\), \(\G[H]_2=\G[H]\), let \(g=\Id\colon
  \G[H]\to\G[H]_2\) and let \(f\colon \G\to\G_2=\DuG[H]\) be the
  bicharacter~\(\bichar\) itself.  Let \(\bichar_2=\Dumultunit[B]\) be
  the reduced bicharacter of~\(\DuG[H]\).  Then
  \[
  (C,\gamma) \boxtimes_\bichar (D,\delta)
  \cong (C,\gamma_2) \boxtimes_{\Dumultunit[B]} (D,\delta),
  \]
  where \(\gamma_2\colon C\to C\otimes\hat{B}\) is the
  \(\DuG[H]\)\nb-coaction associated to~\(\gamma\) by the
  quantum group morphism~\(f\) corresponding to~\(\bichar\).

  This is a special case of
  Theorem~\ref{the:functoriality_qgr_morphism} because the
  bicharacter~\(\Dumultunit[B]\) describes the identity
  morphism on the quantum group~\(\DuG[H]\).  The composition
  of this with~\(f\) gives again~\(f\), so that the
  bicharacter~\(\bichar\) that we get from
  \(\bichar_2=\Dumultunit[B]\) by the above construction is
  indeed the given on.
\end{example}

\begin{example}
  \label{exa:reduce_to_A}
  Let \(\G_2=\G\), \(\G[H]_2=\DuG\), let \(f=\Id\colon
  \G_2\to\G\) and let \(g\colon \G[H]\to\G[H]_2=\DuG\) be the
  dual of the morphism \(\G\to\DuG[H]_2\) associated to the
  bicharacter~\(\bichar\).  Let \(\bichar_2=\multunit[A]\) be
  the reduced bicharacter of~\(\DuG[H]\).  Then
  \[
  (C,\gamma) \boxtimes_\bichar (D,\delta)
  \cong (C,\gamma) \boxtimes_{\multunit[A]} (D,\delta_2),
  \]
  where \(\delta_2\colon D\to D\otimes \hat{A}\) is the
  \(\DuG\)\nb-coaction associated to~\(\delta\) by the quantum
  group morphism~\(g\).
\end{example}

The last example reduces the twisted tensor
product~\(\boxtimes_\bichar\) for an arbitrary bicharacter to
the special case \(\G[H]=\DuG\) and \(\bichar=\multunit[A]\).

\subsection{Functoriality for various kinds of maps}
\label{sec:functoriality}

\(\Cst\)\nb-algebras may be turned into a category using
several types of maps:
\begin{itemize}
\item morphisms (nondegenerate \Star{}homomorphisms \(C_1\to
  \Mult(C_2)\));
\item proper morphisms (nondegenerate \Star{}homomorphisms \(C_1\to
  C_2\));
\item possibly degenerate \Star{}homomorphisms \(C_1\to C_2\);
\item completely positive maps \(C_1\to C_2\);
\item completely positive contractions \(C_1\to C_2\);
\item completely contractive maps \(C_1\to C_2\);
\item completely bounded maps \(C_1\to C_2\).
\end{itemize}
It is well known that the minimal tensor product is functorial
for such maps; that is, two ``maps'' \(f\colon C_1\to C_2\) and
\(g\colon D_1\to D_2\) induce a ``map'' \(f\otimes g\colon
C_1\otimes D_1\to C_2\otimes D_2\), which is determined by
\((f\otimes g)(c\otimes d)\defeq f(c)\otimes g(d)\).  We claim
that the tensor product~\(\boxtimes_\bichar\) is also
functorial for all these kinds of ``maps'' in the following
sense:

\begin{lemma}
  \label{lem:boxtimes_functorial}
  Let \(f\colon (C_1,\gamma_1)\to (C_2,\gamma_2)\) and \(g\colon
  (D_1,\delta_1)\to (D_2,\delta_2)\) be a ``maps'' that are \(\G\)-
  and \(\G[H]\)\nb-equivariant, respectively.  Then there is a unique
  ``map''
  \[
  f\boxtimes_\bichar g\colon C_1\boxtimes_\bichar D_1\to
  C_2\boxtimes_\bichar D_2,\qquad
  \iota_{C_1}(c)\cdot\iota_{D_1}(d)\mapsto
  \iota_{C_2}(f(c))\cdot\iota_{D_2}(g(d)),
  \]
  and \((f,g)\mapsto f\boxtimes_\bichar g\) is a bifunctor.
\end{lemma}

The notion of equivariance for possibly degenerate
\Star{}homomorphisms or completely bounded maps is defined as
in \cite{Baaj-Skandalis:Hopf_KK}*{Definition 1.8}.  The
multiplier algebra is not functorial for such maps, but the
comultiplication morphism \(\gamma_1\colon C_1\to
\Mult(C_1\otimes A)\) takes values in the smaller algebra
\[
\tilde\Mult(C_1\otimes A) =
\tilde\Mult_A(C_1\otimes A) \defeq
\{x\in \Mult(C_1\otimes A) \mid x\cdot (1_C\otimes A)\cup
(1_C\otimes A)\cdot x\subseteq C_1\otimes A\}.
\]
We write a subscript on~\(\tilde\Mult\) to avoid ambiguities:
\(\tilde\Mult(C\otimes D\otimes A)\) could mean either
\(\tilde\Mult_A(C\otimes D\otimes A)\) or \(\tilde\Mult_{D\otimes
  A}(C\otimes D\otimes A)\).

A completely bounded map \(f\colon C_1\to C_2\) induces a
completely bounded, \(\Mult(D)\)-bilinear map
\[
f\otimes\Id_D\colon C_1\otimes D\to C_2\otimes D.
\]
The map \(f\otimes\Id_D\) is completely positive or completely
contractive if~\(f\) is, and a \Star{}homomorphism if~\(f\) is.  Any
\(D\)\nb-bilinear ``map'' \(h\colon C_1\otimes D\to C_2\otimes D\)
extends uniquely to a \(\Mult(D)\)-bilinear ``map'' \(\bar{h}\colon
\tilde\Mult(C_1\otimes D)\to \tilde\Mult(C_2\otimes D)\): for
\(x\in\tilde\Mult(C_1\otimes D)\), there is a unique
\(\bar{h}(x)\in\tilde\Mult(C_2\otimes D)\) with \(\bar{h}(x)\cdot
(1\otimes d) = h(x\cdot (1\otimes d))\) and \((1\otimes d)\cdot
\bar{h}(x) = h((1\otimes d)\cdot x)\) for all \(d\in D\)
because~\(h\) is \(D\)\nb-linear.  On \(\Mult(C_1\otimes D)\), this
definition only works if~\(h\) is \(C_1\otimes D\)-linear, which is
a serious restriction.  After having constructed the extension, we
write \(f\otimes\Id_D\) for the unique extension of
\(f\otimes\Id_D\) to \(\tilde\Mult\) in order not to change our
formulas.  We use this extension to make sense of the equivariance
condition \(\delta\circ f = (f\otimes\Id_A)\circ\gamma\) for
``maps.''

\begin{proof}[Proof of Lemma~\textup{\ref{lem:boxtimes_functorial}}]
  The uniqueness and hence the functoriality is clear because
  the linear span of \(\iota_{C_1}(c)\cdot\iota_{D_1}(d)\) with
  \(c\in C_1\), \(d\in D_1\) is dense in \(C_1\boxtimes_\bichar
  D_1\) and all types of ``maps'' we consider are bounded
  linear.

  We remarked above that ordinary minimal \(\Cst\)\nb-tensor
  products are functorial for ``maps,'' that is, there is a
  well-defined \(\Bound(\Hils)\)-linear ``map''
  \[
  f\otimes g\otimes \Id_{\Comp(\Hils)}\colon
  C_1\otimes D_1\otimes\Comp(\Hils)
  \to C_2\otimes D_2\otimes\Comp(\Hils).
  \]
  We may extend it to a ``map''
  \[
  (f,g)_*\colon \tilde\Mult_{\Comp(\Hils)}(C_1\otimes D_1\otimes\Comp(\Hils)) \to
  \tilde\Mult_{\Comp(\Hils)}(C_2\otimes D_2\otimes\Comp(\Hils)),
  \]
  Lemma~\ref{lem:boxtimes_Podles} implies
  \(C_i\boxtimes_\bichar D_i\subseteq
  \tilde\Mult_{\Comp(\Hils)}(C_i\otimes
  D_i\otimes\Comp(\Hils))\) for \(i=1,2\).

  We claim that the ``map'' \((f,g)_*\) sends
  \(\iota_{C_1}(c)\cdot\iota_{D_1}(d)\) to
  \(\iota_{C_2}(f(c))\cdot\iota_{D_2}(g(d))\); hence it
  restricts to a ``map'' \(f\boxtimes_\bichar g\colon
  C_1\boxtimes_\bichar D_1\to C_2\boxtimes_\bichar D_2\) with
  the required property.  To prove this claim, we look at two cases
  separately.

  First let
  \(f\) and~\(g\) be equivariant morphisms; then \(f\otimes
  g\otimes\Id_{\Comp(\Hils)}\) is a morphism, hence it extends to
  a \Star{}homomorphism between multiplier algebras.  Since
  \(f\) and~\(g\) are equivariant, this canonical extension
  maps \(\gamma_1(c)_{1\pi}\mapsto \gamma_2(c)_{1\pi}\) and
  \(\delta_1(c)_{2\pi}\mapsto \delta_2(c)_{2\pi}\).  Hence it
  maps \(\iota_{C_1}(c)\iota_{D_1}(d)\) to
  \(\iota_{C_2}(f(c))\cdot\iota_{D_2}(g(d))\) as needed.  If
  \(f\) and~\(g\) are proper morphisms, then
  \(\iota_{C_2}(f(c))\cdot\iota_{D_2}(g(d))\in
  C_2\boxtimes_\bichar D_2\) for all \(c\in C_1\), \(d\in
  D_1\), so that \(f\boxtimes_\bichar g\) is a proper morphism
  as well.

  Now let \(f\) and~\(g\) be completely bounded maps; this
  contains the remaining types as special cases.  By definition,
  \begin{align*}
    (f\otimes g\otimes\Id)(c\otimes 1_D\otimes x)\cdot
    (f\otimes g\otimes\Id)(1_C\otimes d\otimes y)
    &= f(c)\otimes g(d) \otimes x\cdot y
    \\&= (f\otimes g\otimes\Id)(c\otimes d\otimes x\cdot y)
  \end{align*}
  for all \(c\in C_1\), \(d\in D_1\), \(x,y\in\Comp(\Hils)\).
  Since \(f\otimes g\otimes\Id\) is bounded linear, this
  implies the partial multiplicativity \((f\otimes
  g\otimes\Id)(x\cdot y) =(f\otimes\Id)(x)_{13}
  (g\otimes\Id)(y)_{23}\) if \(x\in \tilde\Mult(C_1\otimes
  \Comp(\Hils))\), \(y\in \tilde\Mult(C_2\otimes
  \Comp(\Hils))\).  In particular,
  \[
  (f\otimes g\otimes\Id_{\Comp(\Hils)})
  (\iota_{C_1}(c)\cdot \iota_{D_1}(d))
  = (f\otimes g\otimes \Id_{\Comp(\Hils)})(\iota_{C_1}(c))
  \cdot (f\otimes g\otimes\Id_{\Comp(\Hils)})(\iota_{D_1}(d))
  \]
  for all \(c\in C_1\), \(d\in D_1\).  Finally, the
  equivariance of \(f\) and~\(g\) shows that the right hand
  side is \(\iota_{C_2}(f(c))\cdot \iota_{D_2}(g(d))\).
\end{proof}

\begin{proposition}
  \label{pro:functoriality_injective_surjective}
  If \(f\) and~\(g\) are injective morphisms or
  \Star{}homomorphisms, then so is \(f\boxtimes_\bichar g\),
  and vice versa.

  If \(f\) and~\(g\) are surjective \Star{}homomorphisms, then
  so is \(f\boxtimes_\bichar g\), and vice versa.

  Hence \(f\boxtimes_\bichar g\) is bijective if and only if both
  \(f\) and~\(g\) are bijective.
\end{proposition}

\begin{proof}
  If \(f\) and~\(g\) are injective, so is \(f\otimes
  g\otimes\Id_{\Comp(\Hils)}\); hence its extension to
  multipliers is injective, and so is the restriction to
  \(C_1\boxtimes_\bichar D_1\).  Conversely,
  \((f\boxtimes_\bichar g)(\iota_{C_1}(c)\iota_{D_1}(d))\)
  vanishes if \(f(c)=0\) or \(g(d)=0\); hence \(f\) and~\(g\)
  are injective if \(f\boxtimes_\bichar g\) is.

  If \(f\) and~\(g\) are surjective, then elements of the form
  \((f\boxtimes_\bichar g)(\iota_{C_1}(c)\iota_{D_1}(d))=
  \iota_{C_2}(f(c))\iota_{D_2}(g(d))\) are linearly dense in
  \(C_2\boxtimes_\bichar D_2\).  Hence \(f\boxtimes_\bichar g\)
  is surjective as well.  Conversely, suppose that
  \(f\boxtimes_\bichar g\) is surjective.  Then
  \[
  \iota_{C_2}(f(C_1))\cdot \iota_{D_2}(g(D_1))\cdot \Comp(\Hils)_3
  = (C_2\boxtimes_\bichar D_2)\cdot \Comp(\Hils)_3
  = C_2\otimes D_2 \otimes \Comp(\Hils)
  \]
  by Lemma~\ref{lem:boxtimes_Podles}.  We also have
  \(\iota_{C_2}(f(C_1))\iota_{D_2}(g(D_1))\cdot \Comp(\Hils)_3
  \subseteq f(C_1)\otimes g(D_1)\otimes \Comp(\Hils)\).
  Applying slice maps to \(C_2\) and~\(D_2\), we get
  \(f(C_1)=C_2\) and \(g(D_1)=D_2\).
\end{proof}

Now we apply
Proposition~\ref{pro:functoriality_injective_surjective} to
the equivariant embeddings \(\gamma\colon C\to C\otimes A\) and
\(\delta\colon D\to D\otimes B\) provided in
Lemma~\ref{lemm:G_Cst_alg_indentific} to get an embedding
\[
(C,\gamma)\boxtimes_\bichar (D,\delta)
\to
C\otimes D \otimes (A,\Comult)\boxtimes_\bichar (B,\Comult).
\]
Thus we may describe \((C,\gamma)\boxtimes_\bichar (D,\delta)\)
as the crossed product generated by the embeddings
\((\Id\otimes\iota_A)\gamma_{13}\) of~\(C\) and
\((\Id\otimes\iota_B)\delta_{23}\) of~\(D\) into \(C\otimes D
\otimes (A,\Comult)\boxtimes_\bichar (B,\Comult)\).
This description is particularly useful if we know
\((A,\Comult)\boxtimes_\bichar (B,\Comult)\) more explicitly.

\subsection{Exactness on equivariantly semi-split extensions}
\label{sec:exact_semisplit}

\begin{proposition}
  \label{pro:boxtimes_exact_semisplit}
  The functor \(\blank\boxtimes_\bichar D\) maps an extension
  \(C_1\into C_2\prto C_3\) of \(\G\)\nb-\(\Cst\)\nb-algebras
  with a \(\G\)\nb-equivariant completely bounded section to an
  extension of \(\Cst\)\nb-algebras with a completely bounded
  section.  If the section \(C_3\to C_2\) is an equivariant
  \Star{}homomorphism, completely positive or completely
  contractive, then so is the induced section
  \(C_3\boxtimes_\bichar D\to C_2\boxtimes_\bichar D\).
  Analogous statements hold for the functor
  \(C\boxtimes_\bichar \blank\).
\end{proposition}

\begin{proof}
  We have \(C_1\oplus C_3\cong C_2\) in the additive category
  of \(\G\)\nb-equivariant completely bounded maps, using the
  inclusion map \(C_1\to C_2\) and the section \(s\colon C_3\to
  C_2\).  Since \(\blank \boxtimes_\bichar D\) is an additive
  functor, this implies \(C_1\boxtimes_\bichar D \oplus
  C_3\boxtimes_\bichar D \cong C_2\boxtimes_\bichar D\) in the
  category of completely bounded maps.  Thus
  \[
  C_1\boxtimes_\bichar D \to C_2\boxtimes_\bichar D \to
  C_3\boxtimes_\bichar D
  \]
  is an extension of \(\Cst\)\nb-algebras with
  \(s\boxtimes_\bichar \Id_D\) as a completely bounded linear
  section.  This section is again a \Star{}homomorphism,
  completely contractive, or completely positive if~\(s\) is
  so.
\end{proof}

The functor~\(\blank\boxtimes_\bichar D\) cannot be exact for
arbitrary extensions because this already fails for the
commutative minimal tensor product.  An
\(\G[H]\)\nb-\(\Cst\)\nb-algebra~\(D\) deserves to be called
``exact'' if \(\blank\boxtimes_\bichar D\) is an exact functor
for all~\(\G\) and all bicharacters
\(\bichar\in\U(\hat{A}\hot\hat{B})\).  We plan the study this
notion in future work.

\subsection{Functoriality for correspondences}
\label{sec:functoriality_corr}

Next we want to show that~\(\boxtimes_\bichar\) is functorial
for equivariant correspondences.  Recall that a correspondence
between two \(\Cst\)\nb-algebras \(C_1\) and~\(C_2\) is a
Hilbert \(C_2\)\nb-module~\(\Hilm\) with a nondegenerate left
\(C_1\)\nb-action (by adjointable operators).  In this section,
we assume familiarity with Hilbert modules,
see~\cite{Lance:Hilbert_modules}.

We want to show that a \(\G\)\nb-equivariant correspondence
\(\Hilm\colon C_1\to C_2\) and an \(\G[H]\)\nb-equivariant
correspondence \(\Hilm[F]\colon D_1\to D_2\) induce a
correspondence
\[
\Hilm\boxtimes_\bichar \Hilm[F]\colon C_1\boxtimes_\bichar D_1
\to C_2\boxtimes_\bichar D_2
\]
with suitable functoriality properties, including compatibility
with the composition of correspondences: given further
equivariant correspondences \(\Hilm_2\colon C_2\to C_3\) and
\(\Hilm[F]_2\colon D_2\to D_3\), there is a natural isomorphism
of correspondences
\[
(\Hilm\otimes_{C_2} \Hilm_2) \boxtimes_\bichar
(\Hilm[F]\otimes_{C_2} \Hilm[F]_2)
\cong (\Hilm \boxtimes_\bichar\Hilm[F])
\otimes_{C_2\boxtimes_\bichar D_2}
(\Hilm_2\boxtimes_\bichar \Hilm[F]_2).
\]
This also implies that if \(\Hilm\) and~\(\Hilm[F]\) are
equivariant Morita--Rieffel equivalences (that is, full Hilbert
bimodules), then \(\Hilm\boxtimes_\bichar \Hilm[F]\) is a
Morita--Rieffel equivalence.

Quantum group coactions on Hilbert modules are defined by Baaj
and Skandalis in \cite{Baaj-Skandalis:Hopf_KK}*{Definition
  2.2}, but without considering Podle\'s' continuity condition.
Therefore, we add one condition to our definition.

\begin{definition}
  \label{def:cont_coaction_Hilbert_module}
  A \emph{\(\G\)\nb-equivariant Hilbert module} over a
  \(\G\)\nb-\(\Cst\)\nb-algebra \((C,\gamma)\) is a Hilbert
  \(C\)\nb-module~\(\Hilm\) with a coaction \(\epsilon\colon
  \Hilm\to\tilde\Mult(\Hilm\otimes A)\) with the following
  properties:
  \begin{enumerate}
  \item \(\epsilon(\xi)\gamma(c) = \epsilon(\xi c)\) for
    \(\xi\in\Hilm\), \(c\in C\);
  \item \(\gamma(\langle \xi,\eta\rangle_C) = \langle
    \epsilon(\xi),\epsilon(\eta)\rangle_{\tilde\Mult(C\otimes
      A)}\);
  \item \(\epsilon(\Hilm)\cdot (1\otimes A) = \Hilm\otimes A\);
  \item \((1\otimes A) \cdot \epsilon(\Hilm) = \Hilm\otimes
    A\);
  \item \((\epsilon\otimes\Id_A)\epsilon =
    (\Id_{\Hilm}\otimes\epsilon)\epsilon\).
  \end{enumerate}
\end{definition}
Here
\[
\tilde\Mult(\Hilm\otimes A) \defeq
\{T\in\Bound(C\otimes A,\Hilm\otimes A) \mid
(1_{\Hilm}\otimes A)T\cup T(1_C\otimes A) \subseteq
\Hilm\otimes A \},
\]
where~\(\Bound\) means adjointable operators between Hilbert
modules.  Condition~(5) uses canonical extensions of
\(\epsilon\otimes\Id_A\) and \(\Id_A\otimes\epsilon\) to maps
\[
\Bound(C\otimes A,\Hilm\otimes A) \to
\Bound(C\otimes A\otimes A,\Hilm\otimes A\otimes A),
\]
which are described in \cite{Baaj-Skandalis:Hopf_KK}*{Remarque
  2.5}.  The map~\(\epsilon\) is automatically norm-isometric
by \cite{Baaj-Skandalis:Hopf_KK}*{Proposition 2.4}.

Since~\(\gamma\) satisfies the Podle\'s condition and
\(\epsilon(\Hilm) = \epsilon(\Hilm)\cdot\gamma(C)\), our
condition~(3) is equivalent to \(\epsilon(\Hilm)\cdot (C\otimes
A) = \Hilm\otimes A\).  Thus the conditions in
\cite{Baaj-Skandalis:Hopf_KK}*{Definition 2.2} are equivalent
to our conditions (1)--(3) and~(5).

\begin{remark}
  \label{rem:Hilbert_bimodule_automatically_continuous}
  Our definition and the one by Baaj and Skandalis give the
  same definition for Hilbert bimodules and hence the same
  notion of equivariant Morita--Rieffel equivalence (provided
  the \(\Cst\)\nb-algebras involved carry continuous
  coactions).

  A \emph{Hilbert bimodule} between \(C_1\) and~\(C_2\) is both
  a right Hilbert \(C_2\)\nb-module and a left Hilbert
  \(C_1\)\nb-module, such that the left and right module
  structures commute and the inner products satisfy \(\langle
  \xi,\eta\rangle_{C_1} \cdot \zeta = \xi \cdot \langle \eta,
  \zeta\rangle_{C_2}\) for all \(\xi,\eta,\zeta\in\Hilm\).

  The left and right Hilbert module structures both give the
  same multiplier space \(\tilde\Mult(\Hilm\otimes A)\) because
  \(\Comp(\Hils\otimes A)\) maps \(\tilde\Mult(\Hilm\otimes
  A)\) into \(\Hilm\otimes A\).

  A \emph{\(\G\)\nb-equivariant Hilbert bimodule} is a Hilbert
  bimodule with a \(\G\)\nb-coaction \(\epsilon\colon \Hilm\to
  \tilde\Mult(\Hilm\otimes A)\) that satisfies conditions
  (1)--(5) both for the left and the right Hilbert module
  structure.  There is, however, some duplication here.
  Condition~(5) is the same for the left and right Hilbert
  module structure, and the change between left and right
  exchanges conditions (3) and~(4).  Thus if both Hilbert
  module structures satisfy conditions (1)--(3) and~(5), then
  they both satisfy (1)--(5).  Hence the definitions here and
  in~\cite{Baaj-Skandalis:Hopf_KK} give the same notion of
  Hilbert bimodule.
\end{remark}

The \emph{linking algebra} associated to a Hilbert
\(C\)\nb-module~\(\Hilm\) is the algebra of compact operators
on \(C\oplus\Hilm\) with its block decomposition into
\(\Comp(C,C)\cong C\), \(\Comp(C,\Hilm)\cong\Hilm\),
\(\Comp(\Hilm,C)\cong \Hilm^*\) and \(\Comp(\Hilm,\Hilm) =
\Comp(\Hilm)\).  A \(\G\)\nb-coaction on~\(\Hilm\) induces a
coaction \(\gamma'\colon \Comp(C\oplus\Hilm)\to
\Comp(C\oplus\Hilm)\otimes A\) that is compatible with this
block decomposition by
\cite{Baaj-Skandalis:Hopf_KK}*{Proposition 2.7};
\(\gamma'\)~restricts to \(\epsilon\) and~\(\gamma\) on the
blocks \(\Hilm\) and~\(C\) in \(\Comp(C\oplus\Hilm)\).  Under
the assumptions in~\cite{Baaj-Skandalis:Hopf_KK}, this coaction
need not satisfy the Podle\'s condition, even if~\(\gamma\)
does.  Our additional condition~(4) ensures this because it is
equivalent to \(\gamma'(\Hilm^*)\cdot (1\otimes A) =
\Hilm^*\otimes A\), and this implies
\(\gamma'(\Comp(\Hilm))\cdot (1\otimes A) = \Comp(\Hilm)\otimes
A\) because \(\Comp(\Hilm)=\Hilm\cdot\Hilm^*\).  Condition~(3)
and the continuity of~\(\gamma\) give \(\gamma'(\Hilm)\cdot
(1\otimes A) = \Hilm\otimes A\) and \(\gamma'(C)\cdot (1\otimes
A) = C\otimes A\).

\begin{proposition}
  \label{pro:boxtimes_Hilbert_modules}
  Let \(\Hilm\) be a \(\G\)\nb-equivariant Hilbert module over
  \((C,\gamma)\) and let \(\Hilm[F]\) be a
  \(\G\)\nb-equivariant Hilbert module over \((D,\delta)\).
  Let \(C'\defeq \Comp(C\oplus\Hilm)\) and \(D'\defeq
  \Comp(D\oplus\Hilm[F])\) with the induced continuous
  coactions \(\gamma'\) and~\(\delta'\).  Choose a
  \(\bichar\)\nb-Heisenberg pair \((\alpha,\beta)\) and view
  \(C\boxtimes_\bichar D\) and \(\Comp(\Hilm)\boxtimes_\bichar
  \Comp(\Hilm[F])\) as \(\Cst\)\nb-subalgebras of
  \(\Mult(C'\otimes D'\otimes\Comp(\Hils))\).  Then
  \[
  \Hilm \boxtimes_\bichar \Hilm[F] \defeq
  \iota_{C'}(\Hilm)\cdot \iota_{D'}(\Hilm[F]) =
  \iota_{D'}(\Hilm[F]) \cdot \iota_{C'}(\Hilm)
  \]
  is a Hilbert module over \(C\boxtimes_\bichar D\), where the
  right \(C\boxtimes_\bichar D\)-module structure is the
  multiplication in \(\Mult(C'\otimes D'\otimes\Comp(\Hils))\),
  and the \(C\boxtimes_\bichar D\)-valued inner product is
  \(\langle \xi,\eta\rangle \defeq \xi^*\cdot\eta\).  Furthermore,
  \[
  \Comp(\Hilm\boxtimes_\bichar \Hilm[F]) \cong
  \Comp(\Hilm)\boxtimes_\bichar \Comp(\Hilm[F]).
  \]
\end{proposition}

\begin{proof}
  All this follows from
  Lemma~\ref{lem:boxtimes_crossed_product}.
  Lemma~\ref{lem:boxtimes_crossed_product} for \(X=\Hilm\),
  \(Y=\Hilm[F]\) asserts \(\iota_{C'}(\Hilm)\cdot
  \iota_{D'}(\Hilm[F]) = \iota_{D'}(\Hilm[F]) \cdot
  \iota_{C'}(\Hilm)\).  To check that \(\Hilm\boxtimes_\bichar
  \Hilm[F]\) is closed under right multiplication under
  \(C\boxtimes_\bichar D\), we compute
  \begin{align*}
    (\Hilm\boxtimes_\bichar \Hilm[F])
    \cdot (C\boxtimes_\bichar D)
    &= \iota_{C'}(\Hilm) \cdot \iota_{D'}(\Hilm[F])
    \cdot \iota_{D'}(D) \cdot \iota_{C'}(C)
    \\ &= \iota_{C'}(\Hilm) \cdot \iota_{D'}(\Hilm[F])
    \cdot \iota_{C'}(C)
    = \iota_{D'}(\Hilm[F])\cdot
    \iota_{C'}(\Hilm) \cdot \iota_{C'}(C)
    \\ &= \iota_{D'}(\Hilm[F])\cdot \iota_{C'}(\Hilm)
    = \Hilm\boxtimes_\bichar \Hilm[F].
  \end{align*}
  Similar computations give
  \begin{align*}
    (\Hilm\boxtimes_\bichar \Hilm[F])^*
    \cdot (\Hilm\boxtimes_\bichar \Hilm[F])
    &\subseteq C\boxtimes_\bichar D,\\
    (\Hilm\boxtimes_\bichar \Hilm[F])
    \cdot (\Hilm\boxtimes_\bichar \Hilm[F])^*
    &= \Comp(C)\boxtimes_\bichar \Comp(D).
  \end{align*}
  The first line completes the proof that
  \(\Hilm\boxtimes_\bichar \Hilm[F]\) is a Hilbert module over
  \(C\boxtimes_\bichar D\).  The second line says that
  \(\Comp(\Hilm\boxtimes_\bichar \Hilm[F]) \cong
  \Comp(\Hilm)\boxtimes_\bichar \Comp(\Hilm[F])\).
\end{proof}

Let \(\Hilm_1\) and~\(\Hilm_2\) be \(\G\)\nb-equivariant
Hilbert modules over~\(C\) and let \(S\colon
\Hilm_1\to\Hilm_2\) be an adjointable operator.  We want to
construct an induced adjointable operator
\[
S\boxtimes_\bichar \Id_{\Hilm[F]}\colon
\Hilm_1\boxtimes_\bichar\Hilm[F]
\to \Hilm_2\boxtimes_\bichar\Hilm[F].
\]
We may view~\(S\) as an adjointable operator on \(\Hilm\defeq
\Hilm_1\oplus\Hilm_2\) that vanishes on~\(\Hilm_2\) and has
image contained in~\(\Hilm_2\).  There is a canonical unital
\Star{}homomorphism
\[
\Bound(\Hilm) \cong \Mult(\Comp(\Hilm))
\to \Mult(\Comp(\Hilm)\boxtimes_\bichar \Comp(\Hilm[F]))
\cong \Mult(\Comp(\Hilm\boxtimes_\bichar\Hilm[F]))
\cong \Bound(\Hilm\boxtimes_\bichar\Hilm[F]).
\]
We apply it to~\(S\) and then notice that the resulting
operator is the extension by zero of an adjointable operator
\(\Hilm_1\boxtimes_\bichar\Hilm[F]\to\Hilm_2\boxtimes_\bichar\Hilm[F]\).
This defines \(S\boxtimes_\bichar \Id_{\Hilm[F]}\).  The map
\(S\mapsto S\boxtimes_\bichar \Id_{\Hilm[F]}\) is a unital,
strictly continuous \Star{}homomorphism.

A similar construction turns an adjointable operator \(T\colon
\Hilm[F]_1\to\Hilm[F]_2\) between \(\G[H]\)\nb-equivariant
Hilbert \(D\)\nb-modules \(\Hilm[F]_1\) and~\(\Hilm[F]_2\) into
an adjointable operators
\[
\Id_{\Hilm}\boxtimes_\bichar T\colon
\Hilm\boxtimes_\bichar\Hilm[F]_1 \to
\Hilm\boxtimes_\bichar\Hilm[F]_2.
\]

The adjointable operators \(S\boxtimes_\bichar \Id_{\Hilm[F]}\)
and \(\Id_{\Hilm}\boxtimes_\bichar T\) usually do not commute
with each other, so that \(S\boxtimes_\bichar T\) is not
defined unambiguously.  However, if \(S\) and~\(T\) are both
equivariant and unitary, then they are compatible with all
structure that is used to define \(\Hilm_1\boxtimes_\bichar
\Hilm[F]_1\) and \(\Hilm_2\boxtimes_\bichar \Hilm[F]_2\) and
hence must induce an isomorphism
\[
S\boxtimes_\bichar T\colon
\Hilm_1\boxtimes_\bichar \Hilm[F]_1
\to \Hilm_2\boxtimes_\bichar \Hilm[F]_2.
\]

Indeed, the following lemma shows that \(S\boxtimes_\bichar
\Id_{\Hilm[F]}\) and \(\Id_{\Hilm}\boxtimes_\bichar T\) commute
whenever \(S\) or~\(T\) is equivariant.

\begin{lemma}
  \label{lem:equivariant_commute}
  Let \(x\in\Mult(C)\) and \(y\in\Mult(D)\) and assume that
  \(x\)~is \(\G\)\nb-invariant or that \(y\)~is
  \(\G[H]\)\nb-invariant, that is, \(\gamma(x)=x\otimes 1\) or
  \(\delta(y)=y\otimes 1\).  Then \([\iota_C(x),\iota_D(y)]=0\)
  in \(C\boxtimes_\bichar D\).
\end{lemma}

\begin{proof}
  If~\(x\) is \(\G\)\nb-invariant, then \(\iota_C(x)=x\otimes
  1\otimes 1\) in \(\Mult(C\otimes D\otimes\Comp(\Hils))\).
  This commutes with \(\iota_D(y)\in
  \Mult(D\otimes\Comp(\Hils))_{23}\) because it lives in a
  different leg.  The argument for \(\G[H]\)\nb-invariant~\(y\)
  is the same.
\end{proof}

Now we turn from Hilbert modules to correspondences.

A \emph{\(\G\)\nb-equivariant correspondence} from~\(C_1\)
to~\(C_2\) is a \(\G\)\nb-equivariant Hilbert module~\(\Hilm\)
over~\(C_2\) with a nondegenerate representation of~\(C_1\),
that is, with a morphism \(f\colon
C_1\to\Bound(\Hilm)=\Mult(\Comp(\Hilm))\).  Usually, we do not
mention~\(f\) and instead equip~\(\Hilm\) with the left
\(C_1\)\nb-module structure given by~\(f\); thus a
correspondence is a bimodule with a \(C_2\)\nb-valued right
inner product and a \(\G\)\nb-coaction with suitable
properties.

Let~\(\Hilm\) with \(f\colon C_1\to\Bound(\Hilm)\) be a
\(\G\)\nb-equivariant correspondence from~\(C_1\) to~\(C_2\)
and let~\(\Hilm[F]\) with \(g\colon D_1\to\Bound(\Hilm[F])\) be
an \(\G[H]\)\nb-equivariant correspondence from~\(D_1\)
to~\(D_2\).  Then we get a Hilbert module
\(\Hilm\boxtimes_\bichar \Hilm[F]\) over \(C_2\boxtimes_\bichar
D_2\) and an induced morphism
\[
C_1\boxtimes_\bichar D_1 \to
\Mult(\Comp(\Hilm)\boxtimes_\bichar \Comp(\Hilm[F]))
\cong \Mult(\Comp(\Hilm\boxtimes_\bichar\Hilm[F]))
\cong \Bound(\Hilm\boxtimes_\bichar\Hilm[F])
\]
by Proposition~\ref{pro:boxtimes_Hilbert_modules}.  This gives
a correspondence from \(C_1\boxtimes_\bichar D_1\) to
\(C_2\boxtimes_\bichar D_2\).

Let \(S\colon \Hilm_1\to \Hilm_2\) and \(T\colon \Hilm[F]_1\to
\Hilm[F]_2\) be isomorphisms of equivariant correspondences
(that is, equivariant unitaries commuting with the left module
structures).  Then
\[
S\boxtimes_\bichar T\colon
\Hilm_1\boxtimes_\bichar\Hilm[F]_1 \to
\Hilm_2\boxtimes_\bichar\Hilm[F]_2
\]
is an isomorphism of correspondences.  Thus our construction
descends to isomorphism classes of correspondences.

Next we consider the composition of correspondences.
Let~\(C_i\) for \(i=1,2,3\) be \(\G\)\nb-\(\Cst\)-algebras and
let~\(D_i\) for \(i=1,2,3\) be \(\G[H]\)\nb-\(\Cst\)-algebras;
let~\(\Hilm_1\) be a \(\G\)\nb-equivariant correspondence
from~\(C_1\) to~\(C_2\), let~\(\Hilm_2\) be a
\(\G\)\nb-equivariant correspondence from~\(C_2\) to~\(C_3\),
let~\(\Hilm[F]_1\) be an \(\G[H]\)\nb-equivariant
correspondence from~\(D_1\) to~\(D_2\), and let~\(\Hilm[F]_2\)
be an \(\G[H]\)\nb-equivariant correspondence from~\(D_2\)
to~\(D_3\).  The composite correspondences
\(\Hilm_1\otimes_{C_2} \Hilm_2\) and \(\Hilm[F]_1\otimes_{D_2}
\Hilm[F]_2\) are again equivariant (see
\cite{Baaj-Skandalis:Hopf_KK}*{Proposition 2.10}, our extra
continuity condition is easily checked).

\begin{lemma}
  There is a natural isomorphism of correspondences
  \[
  (\Hilm_1\otimes_{C_2} \Hilm_2) \boxtimes_\bichar
  (\Hilm[F]_1\otimes_{D_2} \Hilm[F]_2)
  \cong (\Hilm_1 \boxtimes_\bichar\Hilm[F]_1)
  \otimes_{C_2\boxtimes_\bichar D_2}
  (\Hilm_2\boxtimes_\bichar \Hilm[F]_2).
  \]
\end{lemma}

\begin{proof}
  Let us assume that~\(C_i\) is a subalgebra of
  \(\Bound(\Hilm_i)\) for \(i=1,2\) (we can make these
  representations faithful by taking direct sums with suitable
  correspondences, and then argue in the end that the result
  remains true without these additional summands).  The direct
  sum
  \[
  \Hilm' \defeq C_3 \oplus \Hilm_2 \oplus
  (\Hilm_1\otimes_{C_2} \Hilm_2)
  \]
  is a \(\G\)\nb-equivariant Hilbert \(C_3\)\nb-module on which the
  \(\G\)\nb-\(\Cst\)-algebras \(C_i\) for \(i=1,2,3\) and the
  \(\G\)\nb-equivariant Hilbert modules \(\Hilm_1\), \(\Hilm_2\) and
  \(\Hilm_1\otimes_{C_2}\Hilm_2\) act by adjointable operators.
  Namely, \(C_1\)~acts by the given left action on~\(\Hilm_1\) and by
  zero on the other summands; \(C_2\)~acts by the given left action
  on~\(\Hilm_2\) and by zero on the other summands; \(C_3\)~acts on
  itself by left multiplication and by zero on the other summands;
  \(\Hilm_2\) and \(\Hilm_1\otimes_{C_2}\Hilm_2\) act by the
  isomorphisms \(\Hilm_2\cong\Comp(C_3,\Hilm_2)\) and
  \(\Hilm_1\otimes_{C_2}\Hilm_2\cong\Comp(C_3,\Hilm_1\otimes_{C_2}\Hilm_2)\)
  on~\(C_3\) and by zero on the other summands; \(\Hilm_1\) acts
  on~\(\Hilm_2\) by the map
  \[
  \Hilm_1\to\Bound(\Hilm_2,\Hilm_1\otimes_{C_2}\Hilm_2),
  \qquad
  \xi\mapsto T_\xi,
  \]
  with \(T_\xi(\eta)\defeq \xi\otimes\eta\) for all
  \(\eta\in\Hilm_2\), \(\xi\in\Hilm_1\), and~\(\Hilm_1\) acts
  by zero on the other summands.

  These representations are nicely compatible in the following
  sense: bimodule structures on our Hilbert modules are always
  represented by composition of adjointable operators, and
  inner products are always represented by \(\langle x,y\rangle
  \defeq x^*\circ y\).  Hence they extend to representations of
  the linking algebras \(\Comp(C_2\oplus \Hilm_1)\)
  of~\(\Hilm_1\), \(\Comp(C_3\oplus\Hilm_2)\) of~\(\Hilm_2\),
  and \(\Comp(C_3\oplus \Hilm_1\otimes_{C_2}\Hilm_2)\) of
  \(\Hilm_1\otimes_{C_2}\Hilm_2\).

  Let us assume similarly that
  \(D_i\subseteq\Bound(\Hilm[F]_i)\) for \(i=1,2\), and let us
  embed the \(\G[H]\)\nb-\(\Cst\)-algebras \(D_i\) for
  \(i=1,2,3\) and the \(\G[H]\)\nb-equivariant Hilbert modules
  \(\Hilm[F]_1\), \(\Hilm[F]_2\) and
  \(\Hilm[F]_1\otimes_{D_2}\Hilm[F]_2\) in a similar fashion
  into \(\Bound(\Hilm[F]')\) with
  \[
  \Hilm[F]' \defeq D_3 \oplus \Hilm[F]_2 \oplus
  (\Hilm[F]_1\otimes_{D_2} \Hilm[F]_2).
  \]

  The tensor products \(C_i\boxtimes_\bichar D_i\),
  \(\Hilm_i\boxtimes_\bichar \Hilm[F]_i\) and
  \((\Hilm_1\otimes_{C_2} \Hilm_2)\boxtimes_\bichar
  (\Hilm[F]_1\otimes_{D_2} \Hilm[F]_2)\) are all embedded into
  the multiplier algebra of
  \[
  \Comp(\Hilm') \boxtimes_\bichar \Comp(\Hilm[F]') \cong
  \Comp(\Hilm' \boxtimes_\bichar \Hilm[F]')
  \]
  by Proposition~\ref{pro:functoriality_injective_surjective}
  and Proposition~\ref{pro:boxtimes_Hilbert_modules}.

  The construction in
  Proposition~\ref{pro:boxtimes_Hilbert_modules} also shows
  that, in this representation, the bimodule structures on the
  Hilbert modules \(\Hilm_i\boxtimes_\bichar \Hilm[F]_i\) and
  \((\Hilm_1\otimes_{C_2} \Hilm_2)\boxtimes_\bichar
  (\Hilm[F]_1\otimes_{D_2} \Hilm[F]_2)\) are given by
  composition, and the inner products by \(\langle x,y\rangle
  \defeq x^*\circ y\).  In such a situation, the composite
  correspondence \((\Hilm_1 \boxtimes_\bichar\Hilm[F]_1)
  \otimes_{C_2\boxtimes_\bichar D_2} (\Hilm_2\boxtimes_\bichar
  \Hilm[F]_2)\) is realised concretely as \((\Hilm_1
  \boxtimes_\bichar\Hilm[F]_1) \cdot (\Hilm_2\boxtimes_\bichar
  \Hilm[F]_2)\) with the bimodule structure over
  \(C_1\boxtimes_\bichar D_1\) and \(C_3\boxtimes_\bichar D_3\)
  given by composition and inner product \(\langle x,y\rangle
  \defeq x^*\cdot y\).

  Lemma~\ref{lem:boxtimes_crossed_product} gives
  \begin{align*}
    (\Hilm_1\otimes_{C_2} \Hilm_2) \boxtimes_\bichar
    (\Hilm[F]_1\otimes_{D_2} \Hilm[F]_2)
    &=
    \iota_{\Comp(\Hilm')}(\Hilm_1)\cdot
    \iota_{\Comp(\Hilm')}(\Hilm_2)\cdot
    \iota_{\Comp(\Hilm[F]')}(\Hilm[F]_1)\cdot
    \iota_{\Comp(\Hilm[F]')}(\Hilm[F]_1)
    \\&=
    \iota_{\Comp(\Hilm')}(\Hilm_1)\cdot
    \iota_{\Comp(\Hilm[F]')}(\Hilm[F]_1)\cdot
    \iota_{\Comp(\Hilm')}(\Hilm_2)\cdot
    \iota_{\Comp(\Hilm[F]')}(\Hilm[F]_1)
    \\&= (\Hilm_1 \boxtimes_\bichar\Hilm[F]_1) \cdot
    (\Hilm_2\boxtimes_\bichar \Hilm[F]_2).\qedhere
  \end{align*}
\end{proof}

\subsection{Cocycle conjugacy}
\label{sec:cocycle}

As a special case of equivariant Morita--Rieffel equivalence, we may
change a coaction by a cocycle:

\begin{definition}[\cite{Baaj-Skandalis:Unitaires}*{Definition 0.4}]
  \label{def:cocycle_unitary}
  A \emph{\(\gamma\)\nb-cocycle} is a unitary
  \(u\in\Mult(C\otimes A)\) with
  \begin{equation}
    \label{cond:cocycle_unitary}
    u_{12}(\gamma\otimes\Id_A)u = (\Id_C\otimes\Comult[A])u
    \qquad\text{in }\U(C\otimes A\otimes A).
  \end{equation}
\end{definition}

We can only treat cocycles that satisfy an extra Podle\'s condition:

\begin{lemma}
  \label{lemm:cocycle_induced_coaction}
  Let \(u\in\U(C\otimes A)\) be a \(\gamma\)\nb-cocycle.
  Define a morphism \(\gamma_u\defeq \Ad_u\circ \gamma\colon
  C\to C\otimes A\).  This is a continuous coaction of~\(\G\)
  if and only if
  \begin{equation}
    \label{eq:cocycle_unitary_density}
    \gamma(C)\cdot u^*\cdot(1_C\otimes A)=C\otimes A.
  \end{equation}
\end{lemma}

\begin{proof}
  The morphism~\(\gamma_u\) is faithful because~\(\gamma\) is.
  We check that it is a comodule structure:
  \[
  (\Id_C\otimes\Comult[A])(u\gamma(c)u^*)
  = u_{12}(\gamma\otimes\Id_A)(u\gamma(c))u^*u_{12}^*
  = (\gamma_u\otimes\Id_A)\gamma_u(c)
  \]
  for all \(c\in C\); the first equality uses
  \eqref{cond:cocycle_unitary} and~\eqref{eq:right_coaction}
  for~\(\gamma\); the second equality again
  uses~\eqref{cond:cocycle_unitary} for all \(c\in C\).

  Since \(u\in\U(C\otimes A)\) we have \(u(C\otimes
  A)=C\otimes A\).  Hence~\eqref{eq:cocycle_unitary_density} is
  equivalent to the Podleś condition \(u\gamma(C)u^*\cdot
  (1\otimes A) = C\otimes A\) for~\(\gamma_u\).
\end{proof}

The following result generalises
\cite{Baaj-Skandalis:Unitaires}*{Proposition 7.6}.

\begin{theorem}
  \label{the:cocycle_boxtimes_isomorphism}
  Let~\(u\) be a \(\gamma\)\nb-cocycle and let~\(v\) be a
  \(\delta\)\nb-cocycle.  Assume both satisfy the Podle\'s
  condition~\eqref{eq:cocycle_unitary_density}.  Define the
  coactions \(\gamma_u\) and~\(\delta_v\) as above.
  Then
  \[
  (C,\gamma) \boxtimes_\bichar (D,\delta) \cong
  (C,\gamma_u) \boxtimes_\bichar (D,\delta_v).
  \]
\end{theorem}

This isomorphism is not one of crossed products, that is, it is
not compatible with the embeddings of \(C\) and~\(D\).

\begin{proof}
  Let~\(\Hilm\) be~\(C\) viewed as a Hilbert module over
  itself.  Define the coaction \(\epsilon\colon \Hilm\to
  \tilde\Mult(\Hilm\otimes A)\) by \(\epsilon(c)\defeq u\cdot
  \gamma(c)\).  We claim that this gives a
  \(\G\)\nb-equivariant Hilbert \(C\)\nb-module.  Conditions
  (1)--(3) in Definition~\ref{def:cont_coaction_Hilbert_module}
  are immediate.  Condition~(4) is equivalent
  to~\eqref{eq:cocycle_unitary_density} by taking adjoints,
  and~(5) is equivalent to the cocycle
  condition~\eqref{cond:cocycle_unitary}.

  Since \(\Hilm=C\) as a Hilbert module, the left
  multiplication action of~\(C\) gives an isomorphism \(C\cong
  \Comp(\Hilm)\).  The induced \(\G\)\nb-coaction
  on~\(\Comp(\Hils)\) is, however, not equivalent to~\(\gamma\)
  but to~\(\gamma_u\): \(\gamma_u(c_1)\cdot \epsilon(c_2) =
  \epsilon(c_1c_2)\) for all \(c_1,c_2\in C\).

  Similarly, let~\(\Hilm[F]\) be~\(D\) viewed as a Hilbert module over
  itself, with the \(\G[H]\)\nb-coaction \(\varphi\colon \Hilm[F]\to
  \tilde\Mult(\Hilm[F]\otimes B)\), \(d\mapsto v\cdot \delta(d)\).
  Then \(\Comp(\Hilm[F])\cong D\) with induced coaction~\(\delta_v\).
  Now Proposition~\ref{pro:boxtimes_Hilbert_modules} gives
  \[
  (C,\gamma_u) \boxtimes_\bichar (D,\gamma_v)
  \cong \Comp(\Hilm\boxtimes_\bichar \Hilm[F]).
  \]
  The identity maps \(C\to\Hilm\) and \(D\to\Hilm[F]\) are
  (non-equivariant) unitary operators.  They give unitary
  operators
  \[
  (C,\gamma)\boxtimes_\bichar (D,\delta)
  \to (\Hilm,\epsilon)\boxtimes_\bichar (D,\delta)
  \to (\Hilm,\epsilon)\boxtimes_\bichar (\Hilm[F],\varphi)
  \]
  of Hilbert \((C,\gamma)\boxtimes_\bichar
  (D,\delta)\)-modules.  Conjugating by this unitary gives a
  \(\Cst\)\nb-algebra isomorphism
  \[
  (C,\gamma)\boxtimes_\bichar (D,\delta) \cong
  \Comp(C\boxtimes_\bichar D) \to
  \Comp(\Hilm\boxtimes_\bichar \Hilm[F]).
  \]
  Now compose this with the isomorphism
  \(\Comp(\Hilm\boxtimes_\bichar \Hilm[F]) \cong (C,\gamma_u)
  \boxtimes_\bichar (D,\gamma_v)\).
\end{proof}

We describe the isomorphism above more explicitly.  To simplify
notation, we treat only \(u\) and assume \(v=1\).  The linking
algebra for~\(\Hilm\) is \(\Mat_2(C)\) with the
\(\G\)\nb-coaction
\[
\begin{pmatrix}
  c_{11}&c_{12}\\c_{21}&c_{22}
\end{pmatrix} \mapsto
\begin{pmatrix}
  \gamma(c_{11})& \gamma(c_{12})u^*\\
  u\gamma(c_{21})&u\gamma(c_{22})u^*
\end{pmatrix}.
\]
The upper left and lower right corners are \((C,\gamma)\)
and~\((C,\gamma_u)\), respectively.  Thus
\((C,\gamma)\boxtimes_\bichar (D,\delta)\) and
\((C,\gamma_u)\boxtimes_\bichar (D,\delta)\) are subalgebras of
\(\Mat_2(C)\boxtimes D\).

Conjugation by the partial isometry
\(s=\bigl(\begin{smallmatrix}0&0\\1&0\end{smallmatrix}\bigr)\)
and its adjoint gives isomorphisms between the two corners
\(C\subseteq\Mat_2(C)\).  The strictly continuous extension of
\(\iota_{\Mat_2(C)}\) maps~\(s\) to a partial isometry in
\(\Mat_2(C)\boxtimes_\bichar D\).  Conjugation by this partial
isometry and its adjoint restricts to isomorphisms between
\((C,\gamma)\boxtimes_\bichar (D,\delta)\) and
\((C,\gamma_u)\boxtimes_\bichar (D,\delta)\).

Call a continuous coaction \emph{inner} if it is a
cocycle-twist of the trivial coaction.

\begin{corollary}
  \label{cor:crossed_prod_iso_tensor_prod_inner_action}
  The crossed product \((C,\gamma)\boxtimes_\bichar
  (D,\delta)\) is isomorphic to~\(C\otimes D\) if \(\gamma\)
  or~\(\delta\) is inner.
\end{corollary}

\begin{proof}
  Let \(u\in\Mult(C\otimes A)\) be a cocycle for the trivial coaction
  \(\tau(c)\defeq c\otimes 1\) and let \(\gamma=\tau_u\).  The
  cocycle~\(u\) satisfies~\eqref{eq:cocycle_unitary_density} by
  Lemma~\ref{lemm:cocycle_induced_coaction}.  Now
  Theorem~\ref{the:cocycle_boxtimes_isomorphism} and
  Example~\ref{exa:boxtimes_trivial_coaction} give
  \((C,\gamma)\boxtimes_\bichar D\cong (C,\tau)\boxtimes_\bichar D
  \cong C\otimes D\).  A similar proof works if~\(\delta\) is inner.
\end{proof}

\begin{example}
  \label{exa:tensor_Comp}
  Let \(U^{\Hils}\) and~\(U^{\Hils[K]}\) be corepresentations
  of \(A\) and~\(B\) on Hilbert spaces \(\Hils\)
  and~\(\Hils[K]\).  These are cocycles for the trivial action
  on~\(\Comp(\Hils)\).
  Assume~\eqref{eq:cocycle_unitary_density} to get continuous
  coactions on \(\Comp(\Hils)\) and \(\Comp(\Hils[K])\).  Then
  \[
  \Comp(\Hils)\boxtimes_\bichar \Comp(\Hils[K])
  \cong \Comp(\Hils)\otimes\Comp(\Hils[K])
  \cong \Comp(\Hils\otimes\Hils[K]).
  \]

  This explains the Hilbert space realisation of
  \(C\boxtimes_\bichar D\) in
  Theorem~\ref{the:crossed_prod_for_cov_reps_in_Hilb_sp_lavel}
  in the case where the corepresentations \(U^{\Hils}\)
  and~\(U^{\Hils[K]}\) used there satisfy the technical
  condition~\eqref{eq:cocycle_unitary_density}.  Then we get a
  faithful morphism \(C\boxtimes_\bichar D\to
  \Comp(\Hils)\boxtimes_\bichar \Comp(\Hils[K])\) from
  Proposition~\ref{pro:functoriality_injective_surjective}.
  When we identify \(\Comp(\Hils)\boxtimes_\bichar
  \Comp(\Hils[K])\cong \Comp(\Hils\otimes\Hils[K])\) as above,
  we get a faithful representation of~\(C\boxtimes_\bichar D\)
  on~\(\Hils\otimes\Hils[K]\).
\end{example}

\section{Examples of twisted tensor products}
\label{sec:examples_boxtimes}

We show in Section~\ref{sec:graded} that the skew-commutative
tensor product of \(\Z/2\)\nb-graded \(\Cst\)\nb-algebras is a
special case of our theory.

In Section~\ref{sec:general_group_coactions} we consider the
case where both \(A\) and~\(B\) are duals of locally compact
groups; in particular, this covers the case where \(A\)
and~\(B\) are locally compact Abelian groups.  Here we
understand bicharacters in a classical way, and we show that
\(C \boxtimes_\bichar D\) for any bicharacter is a Rieffel
deformation of \(C\otimes D\).

In Section~\ref{sec:crossed_products}, we treat crossed
products for coactions and construct the dual coaction on a
crossed product using the functoriality
of~\(\boxtimes_\bichar\).

\subsection{Skew-commutative tensor products}
\label{sec:graded}

Let \(\Z/2=\{0,1\}\) be the two-element group.  Let
\(\G=\G[H]\) be \(\Cst(\Z/2)\) with the usual comultiplication.
Thus a \(\G\)\nb-coaction on a \(\Cst\)\nb-algebra~\(C\) is a
\(\Z/2\)\nb-grading: a decomposition \(C=C_0\oplus C_1\) into
involutive, closed, linear subspaces \(C_0\) and~\(C_1\) of
even and odd elements such that
\[
C_i\cdot C_j = C_{i+j \bmod 2},\qquad
C_i^*=C_i.
\]
Equivalently, \(\alpha'(c_0+c_1) \defeq c_0-c_1\) for \(c_i\in
C_i\) defines an involutive \Star{}automorphism of~\(C\).

The \emph{skew-commutative tensor product} of two
\(\Z/2\)\nb-graded \(\Cst\)\nb-algebras \(C\) and~\(D\) is
defined in \cite{Kasparov:Operator_K}*{§2.6} by imposing the
commutation relation that \(c\in C\) and \(d\in D\)
anti-commute if both are odd, and commute if one of them is
even.  This leads to the \Star{}algebra structure
\begin{align*}
  (c_1\hodot d_1)\cdot (c_2\hodot d_2) &\defeq
  (-1)^{\deg(c_2)\cdot \deg(d_1)} c_1c_2\hodot d_1d_2,
  \\
  (c\hodot d)^* &\defeq (-1)^{\deg(c)\cdot \deg(d)} c^*\hodot d^*
\end{align*}
on the algebraic tensor product \(C\hodot D\) of \(C\)
and~\(D\).  The skew-commutative \(\Cst\)\nb-tensor product
\(C\hot D\) is the completion of the \Star{}algebra \(C\hodot
D\) in the \(\Cst\)\nb-norm
\begin{equation}
  \label{eq:norm_hot}
  \norm{x} \defeq
  \sup \frac{(\rho \hot \lambda)(y^*\cdot x^*\cdot x\cdot y)}
  {(\rho \hot \lambda)(y^*\cdot y)}
\end{equation}
over all non-zero elements \(y\in C\hat\odot D\) and all even
states \(\rho\in C^*\) and \(\lambda\in D^*\) (even means
that \(\rho\) and~\(\lambda\) vanish on \(C_1\) and~\(D_1\),
respectively); here the products and adjoints are with respect
to the \Star{}algebra structure on \(C\hodot D\).

The obvious formulas define morphisms \(\iota_C\colon C\to
C\hot D\) and \(\iota_D\colon D\to C\hot D\), so that \(C\hot
D\) is a crossed product of \(C\) and~\(D\).  We want to show
that \(C\hot D\cong C\boxtimes_\bichar D\) for a suitable
bicharacter \(\bichar\in\U(\hat{A}\otimes\hat{A})\).

The dual \(\DuG\) is the group~\(\Z/2\), so that
\(\hat{A}\otimes\hat{B}\cong \Cont(\Z/2\times\Z/2)\) and a
bicharacter~\(\bichar\) is a bicharacter
\(\Z/2\times\Z/2\to\T\) in a more classical sense.  The unique
non-trivial bicharacter is defined by \(\bichar(1,1)=-1\) and
\(\bichar(i,j)=1\) if \(i=0\) or \(j=0\).

\begin{theorem}
  \label{the:boxtimes_skew-commutative}
  Let \(C\) and~\(D\) be \(\Z/2\)\nb-graded
  \(\Cst\)\nb-algebras and let~\(\bichar\) be the non-trivial
  bicharacter in \(\Cont(\Z/2\times\Z/2)\).  Then the crossed
  product \((C\boxtimes_\bichar D,\iota_C,\iota_D)\) of \(C\)
  and~\(D\) is naturally isomorphic to their skew-commutative
  tensor product.
\end{theorem}

\begin{proof}
  A covariant representation of \(C\) is given by a
  \(\Z/2\)\nb-graded Hilbert space
  \(\Hils=\Hils_0\oplus\Hils_1\) and a representation
  \(\varphi\colon C\to\Bound(\Hils)\) with
  \(\varphi(c_i)(\Hils_j)\subseteq \Hils_{i+j}\) for all
  \(i,j\in\Z/2\).  We choose such a faithful covariant
  representation of~\(A\) and a faithful covariant
  representation \(\psi\colon D\to\Bound(\Hils[K])\) on a
  \(\Z/2\)\nb-graded Hilbert space
  \(\Hils[K]=\Hils[K]_0\oplus\Hils[K]_1\).

  Since \(\hat{A}^\univ=\hat{A}\), the unitary~\(Z\) that is
  used in the Hilbert space description of \(C\boxtimes_\bichar
  D\) is described most easily
  by~\eqref{eq:Z_via_universal_lift}.  This gives
  \(Z(\xi\otimes\eta) = -\xi\otimes\eta\) if \(\xi\in\Hils_1\)
  and \(\eta\in\Hils[K]_1\), and \(Z(\xi\otimes\eta) =
  \xi\otimes\eta\) if \(\xi\in\Hils_0\) or
  \(\eta\in\Hils[K]_0\).  Thus \(\Sigma Z\colon
  \Hils\otimes\Hils[K]\to\Hils[K]\otimes\Hils\) is the braiding
  operator from the Koszul sign rule.  The representations
  \(\varphi_1\) and~\(\tilde\psi_2\) in
  Theorem~\ref{the:crossed_prod_for_cov_reps_in_Hilb_sp_lavel}
  are
  \[
  \varphi_1(c)(\xi\otimes\eta)
  = (\varphi(c)\xi)\otimes\eta,\qquad
  \tilde\psi_2(d)(\xi\otimes\eta)
  = (-1)^{\deg(d)\deg(\xi)}\xi\otimes\psi(d)\eta,
  \]
  as expected from the Koszul sign rule.  It remains to show
  that this pair of representations of \(C\) and~\(D\) yields a
  faithful representation of the skew-commutative tensor
  product \(C\hot D\).  It is clear that we get a
  \Star{}representation of~\(C\hodot D\).

  We must show that, for any \(x\in C\hodot D\), its operator
  norm on \(\Hils\otimes\Hils[K]\) is equal to the norm defined
  in~\eqref{eq:norm_hot}.  The GNS-representation for an even
  state \(\rho\colon C\to\C\) on the Hilbert space
  \(L^2(C,\rho)\) is a covariant representation if we let
  \(L^2(C,\rho)_i\) be the closure of~\(C_i\) in
  \(L^2(C,\rho)\).  The direct sum of these GNS-representations
  for all even states is a faithful representation of~\(C\)
  because any state on~\(C_0\) extends to a state on~\(C\) and
  a representation of~\(C\) is faithful once it is faithful
  on~\(C_0\).  Since \(C\boxtimes_\bichar D\) does not depend
  on the covariant representations, we may assume that
  \(\varphi\) and~\(\psi\) are these direct sums of covariant
  GNS-representations of \(C\) and~\(D\), respectively.  The
  resulting representations \(\varphi_1\) and~\(\tilde\psi_2\)
  are block diagonal with respect to the direct sum over the
  even states \(\rho\) and~\(\lambda\), and each block is
  obtained from the GNS-representation for the pair of even
  states \(\rho\) and~\(\lambda\).  The elements \(y\in C\odot
  D\) in~\eqref{eq:norm_hot} form a dense subset of the Hilbert
  space \(L^2(C,\rho)\otimes L^2(D,\lambda)\), and the
  expression in~\eqref{eq:norm_hot} for fixed \(\rho\)
  and~\(\lambda\) is precisely the norm quotient \(\norm{x\cdot
    y}/\norm{y}\), where \(x\cdot y\) is defined using
  \(\varphi_1\odot \tilde\psi_2\).  Hence the norm
  in~\eqref{eq:norm_hot} is exactly the operator norm for a
  particular choice of the covariant representations
  \(\varphi\) and~\(\psi\).
\end{proof}

For instance, there is an isomorphism
\[
\Cst(\Z/2)\boxtimes_\bichar \Cst(\Z/2) \cong \Mat_2(\C),
\]
mapping the generators of the two copies of the group~\(\Z/2\) to the
anti-commuting involutions
\[
g_1 \defeq \begin{pmatrix}
  1&0\\0&-1
\end{pmatrix},\qquad
g_2 \defeq \begin{pmatrix}
  0&1\\1&0
\end{pmatrix}.
\]
Combining this computation with the discussion after
Proposition~\ref{pro:functoriality_injective_surjective} gives an
alternative description of the skew-commutative tensor product: it is
the crossed product generated by the embeddings of \(C\) and~\(D\)
into \(\Mat_2(C\otimes D)\), mapping \(c\mapsto c\otimes 1\otimes 1\)
for even \(c\in C\), \(c\mapsto c\otimes 1\otimes g_1\) for odd \(c\in
C\), \(d\mapsto 1\otimes d\otimes 1\) for even \(d\in D\), \(d\mapsto
1\otimes d\otimes g_2\) for odd \(d\in D\).

\subsection{General group coactions}
\label{sec:general_group_coactions}

Now we consider the case where \(A=\Cred(G)\) and
\(B=\Cred(H)\) for two locally compact groups \(G\) and~\(H\)
with the usual comultiplications.  Thus coactions of \(\G\)
and~\(\G[H]\) are coactions of these groups \(G\) and~\(H\) in
the usual sense.  We are going to identify \(C\boxtimes_\bichar
D\) with a Rieffel deformation of the commutative tensor
product \(C\otimes D\) in the sense
of~\cite{Kasprzak:Rieffel_deformation}.  To begin with, we
reduce to the case where both \(G\) and~\(H\) are Abelian.

The dual quantum groups are the groups \(G\) and~\(H\),
respectively.  Since \(\hat{A}\otimes\hat{B} = \Contvin(G\times
H)\), a bicharacter \(\bichar\in
\U(\hat{A}\otimes\hat{B})\) is a bicharacter
\(\bichar\colon G\times H\to\T\) in the classical sense.
Since~\(\T\) is commutative, \(\bichar(g,h)\) vanishes if \(g\)
or~\(h\) is a commutator.  Hence~\(\bichar\) descends to a
continuous biadditive map \(\bichar'\colon G^\ab\times
H^\ab\to\T\) on the Abelianisations \(G^\ab\) and~\(H^\ab\),
giving us a bicharacter
\(\bichar^\ab\in\U(\Contvin(G^\ab\times H^\ab))\).  The
quotient maps \(G\to G^\ab\) and \(H\to H^\ab\) are quantum
group morphisms.  They allow us to turn the given coactions of
\(G\) and~\(H\) on \(C\) and~\(D\) into coactions of \(G^\ab\)
and~\(H^\ab\), respectively.
Theorem~\ref{the:functoriality_qgr_morphism} shows
\(C\boxtimes_\bichar D = C\boxtimes_{\bichar^\ab} D\), where
the right hand side uses only the induced coactions of
\(G^\ab\) and~\(H^\ab\).  Hence we may without loss of
generality assume that \(G\) and~\(H\) are Abelian locally
compact groups.

Let \(\hat{G}\) and~\(\hat{H}\) be their Pontryagin duals.  We
may also view a bicharacter as a continuous group homomorphism
\(G\to\hat{H}\) or \(H\to\hat{G}\) by fixing one of the two
variables.  This makes it easy to list all bicharacters for two
given Abelian locally compact groups \(G\) and~\(H\).
Coactions of \(G\) and~\(H\) are equivalent to actions of
\(\hat{G}\) and~\(\hat{H}\), respectively.  Thus \(C\)
and~\(D\) carry actions of \(\hat{G}\) and~\(\hat{H}\),
respectively.  The commutative tensor product
\[
E\defeq C\otimes D
\]
inherits an action of \(\Gamma\defeq \hat{G}\times \hat{H}\).
The bicharacter \(\bichar\colon G\times H\to\T\) yields a
bicharacter
\[
\Psi\colon \hat{\Gamma}\times\hat{\Gamma} \to \T,\qquad
\Psi\bigl((g_1,h_1),(g_2,h_2)\bigr) \defeq
\bichar(g_2,h_1)^{-1}.
\]
Any bicharacter is also a two-cocycle, which may be used as a
deformation parameter for Rieffel deformations.  Here we define
Rieffel deformations following
Kasprzak~\cite{Kasprzak:Rieffel_deformation} using crossed
products and Landstad theory.

\begin{theorem}
  \label{the:twisted_tensor_as_Rieffel_def}
  \(C\boxtimes_\bichar D\) is naturally isomorphic to the
  Rieffel deformation of~\(E\) with respect to~\(\Psi\).
\end{theorem}

\begin{proof}
  Pick faithful representations of \(C\rtimes\hat{G}\) and
  \(D\rtimes\hat{H}\) on Hilbert spaces \(\Hils\)
  and~\(\Hils[K]\), respectively.  These give faithful
  covariant representations of \(C\) and~\(D\), which we use to
  represent \(C\boxtimes_\bichar D\) faithfully on
  \(\Hils\otimes\Hils[K]\).  They also generate a faithful
  representation of
  \[
  E\rtimes\Gamma
  \cong (C\rtimes\hat{G}) \otimes (D\rtimes\hat{H})
  \]
  on \(\Hils\otimes\Hils[K]\).  The description of the
  operator~\(Z\) in~\eqref{eq:Z_via_universal_lift} shows that
  \(Z\in\Mult(E\rtimes\Gamma)\).  Thus \(C\boxtimes_\bichar D\)
  is contained in \(\Mult(E\rtimes\Gamma)\), generated by the
  canonical embeddings of \(C\) and~\(D\), with the latter
  twisted by \(\Ad_Z\).

  The Rieffel deformation~\(E^\Psi\) of~\(E\) with respect
  to~\(\Psi\) is described
  in~\cite{Kasprzak:Rieffel_deformation} as a subalgebra of
  \(\Mult(E\rtimes\Gamma)\) as well.  We will use formal
  properties of~\(E^\Psi\) to deduce that \(E^\Psi =
  C\boxtimes_\bichar D\) as \(\Cst\)\nb-subalgebras of
  \(E\rtimes\Gamma\).

  Since \(E=(C\otimes1)\cdot (1\otimes D)\),
  \cite{Kasprzak:Rieffel_coaction}*{Lemma 3.4} yields
  \begin{equation}
    \label{eq:Rieffel_deformation_decompose}
    E^\Psi = (C\otimes 1)^\Psi \cdot (1\otimes D)^\Psi;
  \end{equation}
  Here we let~\(\Gamma\) act on \(C\) and~\(D\) by letting~\(\hat{H}\)
  act trivially on~\(C\) and~\(\hat{G}\) trivially on~\(D\).

  The deformation procedure
  in~\cite{Kasprzak:Rieffel_deformation} uses the unitaries
  \[
  U_{g,h}\in\Cont(G\times H,\T),
  \qquad
  U_{g,h}(g_1,h_1) \defeq \Psi((g_1,h_1),(g,h))
  = \bichar(g,h_1)^{-1}
  \]
  for \((g,h),(g_1,h_1)\in G\times H=\hat{\Gamma}\).
  Since~\(\hat{H}\) acts trivially on~\(C\), we have
  \(C\rtimes\Gamma \cong (C\rtimes \hat{G})\otimes
  \Cst(\hat{H})\).  Hence the unitaries~\(U_{g,h}\) are mapped
  to central elements in \(C\rtimes\Gamma\).  In this case, the
  Rieffel deformation does nothing, that is, \(C^\Psi=C\) as
  subalgebras of \(\Mult(C\rtimes\Gamma)\).  Thus
  \((C\otimes1)^\Psi\)
  in~\eqref{eq:Rieffel_deformation_decompose} is represented
  on~\(\Hils\otimes\Hils[K]\) by \(\varphi_1(C)\).

  Now define another two-cocycle on~\(\hat{\Gamma}\) by
  \(\Psi'((g_1,h_1),(g_2,h_2)) \defeq \bichar(g_1,h_2)\); the
  Rieffel deformation for~\(\Psi'\) involves the unitaries
  \(U'_{g,h}(g_1,h_1) \defeq \Psi'((g_1,h_1),(g,h)) =
  \bichar(g_1,h)\), which are mapped to central elements in
  \(D\rtimes\Gamma \cong \Cst(\hat{G}) \otimes (D\rtimes
  \hat{H})\).  Therefore, \(D^{\Psi'}=D\) as subalgebras of
  \(\Mult(D\rtimes\Gamma)\).

  The two-cocycles \(\Psi\) and~\(\Psi'\) are cohomologous: let
  \(f(g,h)=\bichar(g,h)^{-1}\), then
  \[
  \partial f((g_1,h_1),(g_2,h_2)) \defeq
  \frac{f(g_1g_2,h_1h_2)}{f(g_1,h_1)f(g_2,h_2)}
  = \bichar(g_2,h_1)^{-1}\bichar(g_1,h_2)^{-1};
  \]
  thus \(\partial f\cdot\Psi'=\Psi\).  Now
  \cite{Kasprzak:Rieffel_deformation}*{Lemmas 3.4 and 3.5}
  yield
  \[
  D^\Psi = (D^{\Psi'})^{\partial f} = D^{\partial f} = fDf^*
  \]
  as \(\Cst\)\nb-subalgebras of \(D\rtimes\Gamma\).  Here~\(f\)
  is viewed as a unitary element of \(\Cst(\Gamma)\subseteq
  D\rtimes\Gamma\).

  Equation~\eqref{eq:Z_via_universal_lift} shows that the
  representation \(D\rtimes \Gamma \to
  E\rtimes\Gamma\to\Bound(\Hils\otimes\Hils[K])\) maps~\(f\)
  to~\(Z\).  Thus \((1\otimes D)^\Psi = 1\otimes fDf^*\)
  in~\eqref{eq:Rieffel_deformation_decompose} is represented
  on~\(\Hils\otimes\Hils[K]\) by \(\tilde\psi_2(D)\).  Finally,
  \eqref{eq:Rieffel_deformation_decompose} becomes \(E^\Psi =
  \varphi_1(C)\cdot\tilde\psi_2(D) \cong C\boxtimes_\bichar D\)
  as desired.
\end{proof}

\subsection{Crossed products}
\label{sec:crossed_products}

Consider the special case where \(\G[H]=\DuG\),
\(\bichar=\multunit[A]\in\U(\hat{A}\otimes A)\),
\(D=\hat{A}\), \(\delta=\DuComult[A]\colon \hat{A}\to
\hat{A}\otimes\hat{A}\).  We claim that
\((C,\gamma)\boxtimes_{\multunit[A]} (\hat{A},\DuComult[A])\) is
the \emph{reduced crossed product} of~\((C,\gamma)\).  More
precisely, the reduced crossed product \(C\rtimes_\red
\hat{A}\) comes equipped with canonical morphisms
\(\iota_C\colon C\to C\rtimes_\red \hat{A}\) and
\(\iota_{\hat{A}}\colon \hat{A}\to C\rtimes_\red \hat{A}\),
such that \((C\rtimes_\red \hat{A}, \iota_C,\iota_{\hat{A}})\)
is a crossed product in the sense of
Definition~\ref{def:crossed_product}.  We claim that this is
equivalent to \((C,\gamma)\boxtimes_{\multunit[A]}
(\hat{A},\DuComult[A])\) as a crossed product.

Let \((\pi,\hat{\pi})\) be a \(\G\)\nb-Heisenberg pair on the
Hilbert space~\(\Hils\) of the special form in
Example~\ref{exa:Heisenberg_from_Multunit}; that is,
\(\multunit_{\hat\pi\pi}=\Multunit\) is a multiplicative
unitary generating~\(\G\).

\begin{theorem}
  \label{the:crossed_as_boxtimes}
  There is a faithful morphism \(\rho\colon (C,\gamma)
  \boxtimes_{\multunit[A]} (\hat{A},\DuComult[A]) \to
  C\otimes\Comp(\Hils)\) with
  \(\rho\circ\iota_C=\gamma_{1\pi}\) and
  \(\rho\circ\iota_{\hat{A}}=\hat{\pi}_2\), where
  \begin{alignat*}{2}
    \gamma_{1\pi}&\colon C\to\Mult(C\otimes\Comp(\Hils)),
    &\qquad c&\mapsto (\Id_C\otimes\pi)\gamma(c),\\
    \hat\pi_2&\colon \hat{A}\to\Mult(C\otimes\Comp(\Hils)),
    &\qquad \hat{a}&\mapsto 1_C\otimes\hat\pi(\hat{a}).
  \end{alignat*}
\end{theorem}

We define the reduced crossed product \(C\rtimes_\red \hat{A}\) as the
crossed product generated by the representations \(\gamma_{1\pi}\)
and~\(\hat\pi_2\).  This definition is standard for locally compact
quantum groups with Haar weights, where \(\pi\) and~\(\hat\pi\) are
taken as the regular representations.  For general \(\Cst\)\nb-quantum
groups, Theorem~\ref{the:crossed_as_boxtimes} says that
\(C\rtimes_\red \hat{A}\) does not depend on the choice
of~\(\Multunit\) and that there is an isomorphism of crossed products
\[
(C,\gamma) \boxtimes_{\multunit[A]} (\hat{A},\DuComult[A]) \cong
C\rtimes_\red \hat{A}.
\]

\begin{proof}[Proof of
  Theorem~\textup{\ref{the:crossed_as_boxtimes}}]
  Since
  \[
  \gamma_{1\pi} \otimes \Id_{\Comp(\Hils)}\colon
  C\otimes\Comp(\Hils)
  \to C\otimes \Comp(\Hils)\otimes\Comp(\Hils)
  \]
  is a faithful morphism, the pair of representations
  \((\gamma_{1\pi},\hat\pi_2)\) generates a faithful
  representation of \(C\boxtimes\hat{A}\) if and only the pair
  \(((\gamma_{1\pi} \otimes \Id_{\Comp(\Hils)})\circ
  \gamma_{1\pi}, (\gamma_{1\pi} \otimes
  \Id_{\Comp(\Hils)})\circ \hat\pi_2)\) does so.  We have
  \((\gamma_{1\pi} \otimes \Id_{\Comp(\Hils)})\circ
  \hat\pi_2(\hat{a}) = \hat\pi_3(\hat{a})\) and
  \[
  (\gamma_{1\pi} \otimes \Id_{\Comp(\Hils)}) \gamma_{1\pi}(c)
  = (\gamma\otimes\Id_A)\gamma(c)_{1\pi\pi}
  = (\Id_C\otimes\Comult[A])\gamma(c)_{1\pi\pi}
  = \Multunit
  (\Id\otimes\pi)\gamma(c)_{12} \Multunit{}^*.
  \]
  Let~\(\Sigma_{23}\) be the coordinate flip.  Conjugating both
  representations by the same
  unitary~\(\Sigma_{23}\Multunit{}^*\) gives a unitarily
  equivalent pair of representations.  Hence we may further
  replace \(\gamma_{1\pi}\) and~\(\hat\pi_2\) by the
  representations \(c\mapsto (\Id\otimes\pi)\gamma(c)_{13}\)
  of~\(C\) and
  \[
  \hat{a}\mapsto
  \Sigma_{23}\Multunit{}^* \hat\pi(\hat{a})_3 \Multunit\Sigma_{23}
  = \DuComult(\hat{a})_{\hat\pi\hat\pi}
  \]
  of~\(\hat{A}\); here we use the standard description
  of~\(\DuComult\) in terms of~\(\Multunit\).

  Thus we arrive at the pair of representations \((\Id_C\otimes
  \hat\pi \otimes\Id_{\Comp(\Hils)}) \iota_C\) and
  \((\Id_C\otimes \hat\pi \otimes\Id_{\Comp(\Hils)})
  \iota_{\hat{A}}\) with \(\iota_C=\gamma_{1\pi}\) and
  \(\iota_{\hat{A}} = (\DuComult[A])_{2\hat\pi}\) in
  \(C\otimes\hat{A}\otimes\Comp(\Hils)\).  Since~\(\hat\pi\) is
  faithful, this pair is equivalent to
  \((\iota_C,\iota_{\hat{A}})\).  Since this pair defines the
  crossed product \(C\boxtimes\hat{A}\), we see that
  \((\gamma_{1\pi},\hat\pi_2)\) generates an equivalent crossed
  product as claimed.
\end{proof}

Viewing the reduced crossed product as a special case
of~\(\boxtimes\) gives us more freedom because we may also
tensor \((C,\gamma)\) with other \(\DuG\)\nb-\(\Cst\)-algebras
and use functoriality.  We now describe the dual coaction in
this way, using the functoriality of~\(\boxtimes\).

The comultiplication \(\DuComult\colon \hat{A}\to
\hat{A}\otimes\hat{A}\) is \(\DuG\)\nb-equivariant if~\(\DuG\) coacts
on \(\hat{A}\otimes\hat{A}\) by \(\Id\otimes\DuComult\colon
\hat{A}\otimes\hat{A} \to \hat{A}\otimes\hat{A}\otimes\hat{A}\).  By
the functoriality of~\(\boxtimes\), this equivariant morphism induces
a morphism
\[
\hat{\delta}\colon C\rtimes_\red\hat{A} \cong C \boxtimes \hat{A}
\to C \boxtimes (\hat{A}\otimes\hat{A})
\cong \hat{A}\otimes (C \boxtimes \hat{A})
\cong \hat{A}\otimes (C \rtimes_\red \hat{A});
\]
here we use Lemma~\ref{lemm:crossed_associativity} in the
second variable to pull out the first factor~\(\hat{A}\).

\begin{lemma}
  \label{lem:dual_coaction_nice}
  The map \(\hat{\delta}\colon C\rtimes_\red\hat{A}\to \hat{A}\otimes
  (C\rtimes_\red\hat{A})\) is a continuous left \(\DuG\)\nb-coaction.
\end{lemma}

\begin{proof}
  The comodule property of~\(\hat{\delta}\) follows from the
  coassociativity of~\(\DuComult\) and the funcotriality
  of~\(\boxtimes\).  The map~\(\hat{\delta}\) is faithful by
  Proposition~\ref{pro:functoriality_injective_surjective}.
  The Podle\'s condition for~\(\hat{\delta}\) follows because
  \((\hat{A}\otimes1)\DuComult(\hat{A})=\hat{A}\otimes\hat{A}\):
  apply~\(\iota_{\hat{A}\otimes\hat{A}}\) to this equality.
\end{proof}

This coaction is uniquely determined by the conditions
\(\hat{\delta}(\iota_C(c)) = 1\otimes\iota_C(c)\) and
\(\hat{\delta}(\iota_{\hat{A}}(\hat{a})) =
(\Id_{\hat{A}}\otimes\iota_{\hat{A}})\DuComult\).  The same
conditions characterise the \emph{dual coaction}.  Thus we have
indeed constructed the dual coaction.

The functoriality of~\(\boxtimes\) in the first variable gives
us the usual functoriality of reduced crossed products.

General tensor products \(C\boxtimes_{\multunit} (D,\delta)\)
are closely related to the crossed product through
Lemma~\ref{lemm:G_Cst_alg_indentific}, which shows that
\(\delta\colon D\to D\otimes\hat{A}\) is a
\(\DuG\)\nb-equivariant embedding for the coaction
\(\Id_D\otimes\DuComult\) on \(D\otimes\hat{A}\).  By
Proposition~\ref{pro:functoriality_injective_surjective} and
Lemma~\ref{lemm:crossed_associativity}, this induces a faithful
morphism
\[
C\boxtimes_{\multunit} D \to
C\boxtimes_{\multunit} (D\otimes\hat{A}) \cong
D\otimes (C\boxtimes_{\multunit} \hat{A}) \cong
D\otimes (C\rtimes_\red \hat{A}).
\]

Now we consider once again the general situation of two quantum
groups \(\Qgrp{G}{A}\) and \(\Qgrp{H}{B}\) and a bicharacter
\(\bichar\in\U(\hat{A}\otimes\hat{B})\).

\begin{theorem}
  \label{the:boxtimes_in_reduced}
  View \(\hat{A}\otimes\hat{B}\) as a subalgebra of
  \((C\rtimes_\red\hat{A})\otimes (D\rtimes_\red\hat{B})\) via
  \(\iota_{\hat{A}}\otimes\iota_{\hat{B}}\) and use this to
  view~\(\bichar\) as a multiplier of \((C\rtimes_\red\hat{A})\otimes
  (D\rtimes_\red\hat{B})\).  The embeddings
  \begin{alignat*}{2}
    (\iota_C)_1&\colon C\to
    (C\rtimes_\red\hat{A})\otimes (D\rtimes_\red\hat{B}),&\qquad
    c&\mapsto \iota_C(c)\otimes 1,\\
    \Ad_{\bichar^*}\circ(\iota_D)_2&\colon D\to
    (C\rtimes_\red\hat{A})\otimes (D\rtimes_\red\hat{B}),&\qquad
    d&\mapsto \bichar^* (1\otimes \iota_D(d))\bichar,
  \end{alignat*}
  induce a faithful morphism
  \[
  C\boxtimes_\bichar D\to
  (C\rtimes_\red\hat{A})\otimes (D\rtimes_\red\hat{B}).
  \]
\end{theorem}

\begin{proof}
  Choose faithful representation \(\varphi_0\colon
  C\to\Bound(\Hils_0)\).  Let \((\pi,\hat\pi)\) be a
  \(\G\)\nb-Heisenberg pair as in
  Example~\ref{exa:Heisenberg_from_Multunit}, acting on a
  Hilbert space~\(\Hils_\pi\).  Let \(\Hils\defeq
  \Hils_0\otimes\Hils_\pi\).  Then we get a faithful
  representation \(\varphi_0\otimes\Id\) of
  \(C\otimes\Comp(\Hils_\pi)\) on~\(\Hils\).  This restricts to
  a faithful representation \(\varphi'\colon
  C\rtimes_\red\hat{A}\subseteq
  C\otimes\Comp(\Hils_\pi)\to\Bound(\Hils)\), where we realise
  \(C\rtimes_\red\hat{A}\) as in
  Theorem~\ref{the:crossed_as_boxtimes}.

  We compare this with the construction of a covariant representation
  of~\((C,\gamma)\) in
  Example~\ref{exa:existence_of_faithful_covariant_representation}.
  We see that this covariant representation consists of
  \(\rho\circ\iota_C\colon C\to\Bound(\Hils)\) and
  \(\multunit[A]_{\rho\iota_{\hat{A}}2}\in \U(\Comp(\Hils_\pi)\otimes
  A)\).  Furthermore, the representation of~\(\hat{A}^\univ\) used
  later in the proof of Theorem~\ref{the:V-Heis_Comm_corep} is
  \(\rho\iota_{\hat{A}}\circ \Lambda\) for the reducing morphism
  \(\Lambda\colon \hat{A}^\univ\to\hat{A}\).  (Actually, any
  representation of \(C\rtimes_\red \hat{A}\) gives a covariant
  representation of \((C,\gamma)\) in a similar way.)

  Now do the same things for \((D,\delta)\): let \(\psi_0\colon
  D\to\Bound(\Hils[K]_0)\) be a faithful representation; choose
  an \(\G[H]\)\nb-Heisenberg pair \((\rho,\hat\rho)\) as in
  Example~\ref{exa:Heisenberg_from_Multunit}, acting on a
  Hilbert space~\(\Hils[K]_\rho\); let \(\Hils[K]\defeq
  \Hils[K]_0\otimes\Hils[K]_\rho\); let~\(\psi'\) be the
  resulting faithful representation of \(D\rtimes_\red\hat{B}\)
  on~\(\Hils[K]\); construct a covariant representation of
  \((D,\delta)\) on~\(\Hils[K]\) as in
  Example~\ref{exa:existence_of_faithful_covariant_representation}.

  Theorem~\ref{the:crossed_prod_for_cov_reps_in_Hilb_sp_lavel}
  gives a faithful representation of \(C\boxtimes_\bichar D\)
  on \(\Hils\otimes\Hils[K]\), generated by the representations
  \(\varphi_1\) and \(\Ad_Z\psi_2\).  By construction, we also
  get a faithful representation \(\varphi'\otimes\psi'\) of
  \((C\rtimes_\red \hat{A})\otimes (D\rtimes_\red\hat{B})\) on
  \(\Hils\otimes\Hils[K]\).  The description of~\(Z\)
  in~\eqref{eq:Z_via_universal_lift} yields
  \(Z=(\varphi'\otimes\psi')(\bichar^*)\).  Hence the
  representations \(\varphi_1\) and \(\Ad_Z\psi_2\) both factor
  through the embedding~\(\varphi'\otimes\psi'\) and the maps
  \((\iota_C)_1\) and \(\Ad_{\bichar^*}\circ(\iota_D)_2\) in
  the statement of the theorem.  We thus get a faithful
  morphism \(C\boxtimes_\bichar D \to (C\rtimes_\red
  \hat{A})\otimes (D\rtimes_\red\hat{B})\) restricting to
  \((\iota_C)_1\) and \(\Ad_{\bichar^*}\circ(\iota_D)_2\) on
  \(C\) and~\(D\), respectively.
\end{proof}

For instance, in the situation of Section~\ref{sec:graded},
this realises the skew-commutative tensor product \(C\otimes
D\) as a subalgebra of \((C\rtimes\Z/2)\otimes
(D\rtimes\Z/2)\).

\end{document}